\documentclass[11pt]{amsart}

\usepackage{amsmath,amsthm,amssymb,mathtools}
\usepackage{color,tikz-cd,hyperref,enumitem}
\hypersetup{colorlinks=true,linkcolor=blue,citecolor=blue,urlcolor=blue,linktocpage}
\usetikzlibrary{patterns}
\numberwithin{equation}{section}
\newtheorem{thm}{Theorem}[section]
\newtheorem{prp}[thm]{Proposition}
\newtheorem{lem}[thm]{Lemma}
\newtheorem{cor}[thm]{Corollary}

\newtheorem{thm-intro}{Theorem}
\newtheorem{cor-intro}[thm-intro]{Corollary}

\theoremstyle{definition}
\newtheorem{dfn}[thm]{Definition}

\newtheorem{rmk}[thm]{Remark}
 
\theoremstyle{remark}

\newcommand{\vb}{\,|\,}

\newcommand{\Z}{\mathbb{Z}}

\newcommand{\R}{\mathbb{R}}

\newcommand{\cA}{\mathcal{A}}
\newcommand{\cB}{\mathcal{B}}

\newcommand{\cC}{\mathcal{C}}
\newcommand{\cD}{\mathcal{D}}

\newcommand{\cW}{\mathcal{W}}

\newcommand{\cM}{\mathcal{M}}

\newcommand{\colim}{\textup{colim}}

\newcommand{\hocolim}{\textup{hocolim}}

\newcommand{\Perf}{\textup{Perf}}

\newcommand{\Tw}{\textup{Tw}}

\newcommand{\Cyl}{\textup{Cyl}}

\newcommand{\dgCat}{\textup{dgCat}}
\newcommand{\dgAlg}{\textup{dgAlg}}
\newcommand{\sCat}{\textup{dgCat}_s}

\newcommand{\Ob}{\mathrm{Ob}}
\newcommand{\Mor}{\mathrm{Mor}}

\newcommand{\Ho}{\textup{Ho}}

\title{A computational approach to the homotopy theory of dg categories}
\author{Dogancan Karabas and Sangjin Lee}

\begin{document}
	
\begin{abstract}
	We give a specific cylinder functor for semifree dg categories. This allows us to construct a homotopy colimit functor explicitly. 
	These two functors are ``computable", specifically, the constructed cylinder functor sends a dg category of strictly finite type, i.e., a semifree dg category having finitely many objects and generating morphisms, to a dg category of strictly finite type.
	The homotopy colimit functor has a similar property.  
	Moreover, using the cylinder functor, we give a cofibration category of semifree dg categories and that of dg categories of strictly finite type, independently from the work of Tabuada \cite{tabuada-model}.
	All the results similarly work for semifree dg algebras.
	We also describe an application to symplectic topology and provide a toy example.
\end{abstract}
	
\date{\today}
\maketitle
\tableofcontents

\section{Introduction}
\label{section introduction}

Homotopy theory, originating in algebraic topology, plays a pivotal role in numerous areas of modern mathematics. Specifically, the homotopy theory of differential graded (dg) categories is prominently featured in fields such as algebraic geometry, representation theory, higher categories, and symplectic topology. To explore their homotopy theory up to various weak equivalences, Tabuada \cite{tabuada-model,tabuada-model-morita} introduced {\em model structures} for the category of dg categories. These model structures come with auxiliary functors called cofibrations and fibrations, and provide two {\em functorial factorizations} of functors: the first type factors a functor into a cofibration followed by a weak equivalence, and the second type factors it into a weak equivalence followed by a fibration.

There is another approach to the homotopy theory of dg categories without using a model structure: Starting with cofibrations, instead of considering the entire factorization data, one can focus solely on the functorial factorization of codiagonal maps into a cofibration and a weak equivalence. This functorial factorization yields a construction of a {\em cylinder functor} (Definition \ref{dfn:cylinder-functor}). Once additional axioms (see Definition \ref{dfn:i-category}) are established, this construction recovers an entire functorial factorization of the first type, representing ``half" of a model structure known as a {\em cofibration category} in the sense of \cite{baues}. Therefore, constructing a cylinder functor allows us to delve into the homotopy theory of dg categories, specifically allowing us to describe the {\em homotopy colimit functor} on diagrams of dg categories.

In the current paper, we present a {\em simple construction of a cylinder functor} on the category of dg categories, expanding upon our earlier work \cite{hocolim} where the cylinder functor was defined solely at the level of objects.
Moreover, our cylinder functor leads us to a {\em simple homotopy colimit functor} and establishes an {\em $I$-category} and a {\em cofibration category} of dg categories, offering a computational approach to the homotopy theory of dg categories.
Detailed results can be found in Theorems \ref{thm-intro:cyl}--\ref{thm-intro:model}, which will be discussed following the introduction of our framework. We will also comment on the simplicity of our constructions.

Throughout this paper, we mostly work within the context of the category of semifree dg categories over a commutative ring $k$, denoted by $\sCat$ (see Definition \ref{dfn:semifree}). We note that every dg category has a semifree resolution (as discussed in \cite{drinfeld}), emphasizing the significance of our focus on $\sCat$. Furthermore, we fix weak equivalences as quasi-equivalences, pretriangulated equivalences, or Morita equivalences.

\begin{thm-intro}\label{thm-intro:cyl}[Corollary \ref{cor:cyl-functor}]
	Let $\Cyl\colon\dgCat_s\to\dgCat_s$ be a functor together with the natural transformations $i_1,i_2\colon 1_{\dgCat_s}\Rightarrow \Cyl$ and $p\colon\Cyl\Rightarrow 1_{\dgCat_s}$ defined as follows:
	\begin{itemize}
		\item $\Cyl(\cC)$ for any $\cC\in\dgCat_s$ and the natural transformations are as in Theorem \ref{thm:cylinder-object},
		\item $\Cyl(F)$ for any morphism (dg functor) $F\colon\cA\to\cB$ in $\dgCat_s$ is as in Theorem \ref{thm:cylinder-morphism}.
	\end{itemize}
	Then, $\Cyl$ is a cylinder functor, which means the following conditions are satisfied:
	\begin{itemize}
			\item $p\circ (i_1\amalg i_2)\colon \cC\amalg\cC\to \Cyl(\cC)\to \cC$ is the codiagonal of $\cA$,
			\item $i_1\amalg i_2\colon\cC\amalg\cC\to\Cyl(\cC)$ is a cofibration,
			\item $p\colon\Cyl(\cC)\to\cC$ is a weak equivalence and a fibration.
	\end{itemize}
\end{thm-intro}
We note that our cylinder functor differs significantly from the natural cylinder functor induced by Tabuada's model structures. Tabuada employed Quillen's ``small object argument" (as detailed in Hovey \cite{hovey}) to establish a functorial factorization, which results in a cylinder functor as a special case. However, this method (or its refinement by Garner \cite{garner}) involves a transfinite construction that is overly complicated from a computational standpoint.

In contrast, our cylinder functor $\Cyl$ ensures that the number of generating morphisms of a semifree dg category 
$\cC$ and its image $\Cyl(\cC)$ are close in size. Specifically, if the former is finite, the latter is also finite, a property that does not hold for Tabuada's natural cylinder. This effectiveness arises from the construction of $\Cyl$, which is specifically tailored for semifree dg categories. Moreover, the computability of our cylinder functor ensures that all our functorial factorizations are computable as can be seen in Theorem \ref{thm:mapping-cylinder-factorisation}.

Using the cylinder functor described earlier, we can construct the homotopy colimit (pushout) functor for the category of dg categories, where weak equivalences are defined as quasi-equivalences, pretriangulated equivalences, or Morita equivalences:
\begin{thm-intro}\label{thm-intro:hocolim}[Theorem \ref{thm:hocolim-functor-dg}]
	Let $J$ be a category of the form $a\leftarrow c\rightarrow b$, and $\dgCat^J$ denote the category of $J$-diagrams in $\dgCat$. Then, the homotopy colimit functor
	\[\hocolim\colon\Ho(\dgCat^J)\to\Ho(\dgCat)\]
	can be given such that $\hocolim(\cC)$ for any $\cC\in\dgCat_s$ is defined as in Theorem \ref{thm:hocolim-functor-dg}\eqref{item:hocolim-obj}, and $\hocolim(F)$ for any morphism (dg functor) $F\colon\cA\to\cB$ in $\dgCat_s$ is defined as in Theorem \ref{thm:hocolim-functor-dg}\eqref{item:hocolim-mor}.
\end{thm-intro}
Our homotopy colimit functor is simple in the sense that it produces a semifree dg category of a size close to the total sizes of semifree dg categories in a given diagram. Specifically, the finiteness of the number of generators is preserved under our homotopy colimit functor. In contrast, if we compare this to the Grothendieck construction for homotopy colimit functor (refer to \cite{gps2}), our method offers several computational advantages. The Grothendieck construction is not explicit; rather, it is expressed as a localization of a non-semifree dg category, which often leads to computational difficulties.

We note that the description in Theorem \ref{thm-intro:hocolim} still holds for the diagrams of the form $\cA\leftarrow \cC\rightarrow \cB$, where $\cA$ and $\cB$ are not necessarily semifree, if $k$ has flat dimension zero (e.g., if $k$ is a field). See Remark \ref{rmk:hocolim-non-semifree} for details.

Moreover, the constructions in Theorems \ref{thm-intro:cyl} and \ref{thm-intro:hocolim} can be adapted for scenarios where the input of the functors is a localization of a semifree dg category. Please refer to Theorems \ref{thm:cylinder-loc}, \ref{thm:cylinder-morphism-loc}, and \ref{thm:hocolim-functor-dg-loc} for more details on these adjustments. The usage of Theorem \ref{thm:hocolim-functor-dg-loc} is demonstrated through computations for the case $n=2$ in Sections \ref{sec:wfuk-sphere} and \ref{sec:reflection}.

As mentioned earlier, the constructed cylinder functor can define a ``half" of a model structure, or more precisely, a cofibration category structure on $\dgCat_s$. This structure differs from that induced by \cite{tabuada-model} due to its simpler functorial factorization:
\begin{thm-intro}\mbox{}
	\label{thm-intro:model}
	\begin{enumerate}
		\item (Theorem \ref{thm:dgcat-I-category}) The category of semifree dg categories $\textup{dgCat}_s$ is an $I$-category with the structure $(\mathrm{cof},I)$, which is defined as follows:
		\begin{itemize}
			\item $\mathrm{cof}:$ Cofibrations are semifree extensions.
			
			\item $I$: The functor $I$ is the cylinder functor $\Cyl\colon\textup{dgCat}_s \to\textup{dgCat}_s$ from Theorem \ref{thm-intro:cyl}.
		\end{itemize}
		
		\item (Theorem \ref{thm:dgcat-cofibration-category}) The category of semifree dg categories $\textup{dgCat}_s$ forms a cofibration category, where weak equivalences are quasi-equivalences and cofibrations are semifree extensions. Moreover, every object in $\textup{dgCat}_s$ is both fibrant and cofibrant. In particular, $\textup{dgCat}_s$ makes a category of cofibrant objects.
	\end{enumerate}
\end{thm-intro}
Theorems \ref{thm-intro:cyl}--\ref{thm-intro:model} hold in the setting of semifree dg algebras with slight modifications, see Remarks \ref{rmk:cylinder-functor-dg-algebras} and \ref{rmk:hocolim-dg-algebras}, and Theorem \ref{thm:i-cof-cat-dg-algebras}.

We note that Theorem \ref{thm-intro:model} is not necessary for constructions such as Theorem \ref{thm-intro:hocolim}, as we can use our cylinder functor within Tabuada's model structures to carry out homotopy theory. However, we have established additional axioms for the constructed cylinder functor to explicitly construct a cofibration category of semifree dg categories, which is distinct from the approach in \cite{tabuada-model}. In \cite{tabuada-model}, Quillen's small object argument is employed to construct a model structure and establish the existence of factorizations, but they are not effective or concretely expressible in practice. In contrast, our cylinder functor is explicit and computable, allowing for explicit and computable constructions within the cofibration category $\sCat$. Consequently, we can derive Theorem \ref{thm-intro:hocolim} from Theorem \ref{thm-intro:model} without relying on \cite{tabuada-model}.

Moreover, we can directly restrict the cofibration category structure given in Theorem \ref{thm-intro:model} to dg categories of strictly finite type, i.e., semifree dg categories with finitely many objects and generating morphisms, as the factorizations preserve finiteness. See Remark \ref{rmk:finite-type}. The same is not true for the model structure of Tabuada \cite{tabuada-model}, as its given factorizations do not preserve finiteness.
We note that dg categories of strictly finite type are, in particular, of finite type. 
Refer to \cite{dgcat} and \cite{Kontsevich09} for the definition of dg categories of finite type.
	
One can find a direct application of the above results in symplectic topology.

Symplectic manifolds are associated with a powerful symplectic invariant known as the {\em Fukaya category}, defined as an $A_{\infty}$-category in \cite{Fukaya-Oh-Ohta-Ono09a, Fukaya-Oh-Ohta-Ono09b}.
Computing the Fukaya category is typically a notoriously challenging task.
However, Ganatra, Pardon, and Shende \cite{gps1, gps2} proved that for certain open symplectic manifolds (referred as {\em Weinstein manifolds}), the (wrapped) Fukaya category can be computed, after a choice of a covering, as a homotopy colimit of the wrapped Fukaya categories of the covering elements.

Furthermore, given the relation of wrapped Fukaya categories to microlocal sheaves by \cite{gps3}, Nadler \cite{arboreal,nonchar} has shown that the wrapped Fukaya category of each covering element can be regarded as a semifree dg category with finitely many objects and generating morphisms (a dg category of strictly finite type) up to pretriangulated (or Morita) equivalence. This result highlights the significance of Theorem \ref{thm-intro:hocolim} in facilitating such computations.
	
Now, let us consider a symplectomorphism $\phi: W_1 \to W_2$ between two symplectic manifolds. 
It is known that $\phi$ induces a functor between Fukaya categories of $W_1$ and $W_2$. 
If $\phi$ respects the homotopy colimit diagrams that compute Fukaya categories of $W_1$ and $W_2$, one can obtain a specific description of the induced functor using Theorem \ref{thm-intro:hocolim}.
For a detailed example and the symplectic topological motivations behind this discussion, please refer to Section \ref{sec:sphere-reflection}.

	\subsection{Acknowledgment}
	\label{subsection acknowledgment}
	We are grateful to Ezra Getzler for his insightful comments, which inspired the content of Section \ref{sec:cofibration-category}. 
	
	The first-named author is supported by World Premier International Research Center Initiative (WPI), MEXT, Japan. The second-named author is supported by a KIAS Individual Grant (MG094401) at Korea Institute for Advanced Study.

\section{Preliminaries on dg categories}
\label{section preliminaries}

A {\em differential graded (dg) category} is a category enriched over the symmetric monoidal category of complexes over a fixed commutative ring $k$.
It can also be viewed as an $A_{\infty}$-category in which compositions of order greater than 2 are set to vanish.
For further details on dg categories, readers may refer to \cite{dgcat}, and for a review of $A_{\infty}$-categories, one can consult \cite{seidel}.
We use $d$ for the differential and
$\circ$ for compositions of morphisms, and we omit the latter whenever it is convenient.
When introducing a dg category, we follow the convention of providing the following five items:
\begin{enumerate}[label = (\roman*)]
	\item {\em Objects:} We list the objects in the category.
	\item {\em Generating morphisms:} We give a set of generating morphisms. They generate all the morphisms as an algebra, not as a module. We will not explicitly mention the existence of identity morphisms, but it should be understood that every object has the identity endomorphism.
	\item {\em Degrees:} For each generating morphism, we specify its degree.
	\item {\em Differentials:} For each generating morphism, we specify its differential.
	\item {\em Relations:} We specify the relations between generating morphisms. This item will be omitted if the generating morphisms freely generate all other morphisms.
\end{enumerate}

Given a dg category $\cC$, we denote by $\Ob\,\cC$ (or simply $\cC$ when it is clear from the context) the collection of objects in $\cC$, and by $\Mor\,\cC$ the collection of morphisms in $\cC$. We use $\hom^*_{\cC}(A,B)$ (or simply $\hom^*(A,B)$) to represent the cochain complex of morphisms between the objects $A$ and $B$ of $\cC$.

Next, we introduce a specific class of dg categories and dg functors that will be fundamental throughout the paper. For further details, readers can refer to \cite{hocolim}.

\begin{dfn}\label{dfn:semifree}\mbox{}
	\begin{enumerate}
		\item A (small) dg category $\cC$ is called a {\em semifree dg category} if its morphisms, treated as an algebra, are freely generated by a set of morphisms $\{f_i\}$ (indexed by an ordinal), with the condition that $df_i$ is generated by the set $\{f_j\vb j<i\}$. In this case, $\{f_i\}$ is called a set of {\em generating morphisms} of $\cC$.
		
		\item\label{item:semifree-extension} A dg functor $F\colon \cC\to\cD$ is called a {\em semifree extension} by a set of objects $R$ and a set of morphisms $S=\{f_i\}$ if it satisfies the following conditions:
		\begin{itemize}
			\item $F$ is an inclusion.
			
			\item The objects of $F(\cC)$, along with $R$, form the objects of $\cD$.
			
			\item The morphisms of $\cD$, treated as an algebra, can be expressed as a free extension of the morphisms of $F(\cC)$ by $\{f_i\}$ (indexed by an ordinal), with the condition that $df_i$ is generated by the morphisms of $F(\cC)$ and $\{f_j\vb j<i\}$.
		\end{itemize}
		
		\item A dg category $\cD$ is called a {\em semifree extension} of a dg category $\cC$ by a set of objects $R$ and a set of morphisms $S$ if there exists a semifree extension $F\colon\cC\to\cD$ as in Definition \ref{dfn:semifree}\eqref{item:semifree-extension}.
	\end{enumerate}
\end{dfn}

Let $\dgCat$ denote the category of small dg categories, where morphisms are dg functors. We aim to invert certain dg functors, referred to as {\em weak equivalences}, in $\dgCat$. The resulting categories can be studied by introducing {\em model structures} on $\dgCat$, making $\dgCat$ a {\em model category}. See \cite{dwyer-spalinski}, \cite{hovey}, and \cite{hirschhorn} for a review of model categories. More precisely, upon inverting weak equivalences in $\dgCat$, we obtain the {\em homotopy category} $\Ho(\dgCat)$ of the model category $\dgCat$. Refer to Section \ref{subsec:hocolim-via-model} for further details.

For a given dg category $\cC$, we introduce the following notations:
\begin{itemize}
	\item $\Tw\,\cC$ is the dg category of twisted complexes in $\cC$, which is a pretriangulated envelope of $\cC$.
	
	\item $\Perf\,\cC$ is the split-closure (or idempotent completion) of $\Tw\,\cC$.
\end{itemize}
See \cite{seidel} for more details. With these notations, we can state the following theorem:

\begin{thm}[\cite{tabuada-model,tabuada-model-morita}]\label{thm:dgcat-model} The category $\dgCat$ admits the following model structures:
	\begin{enumerate}
		\item \label{item:dwyer-kan-model} {\em Dwyer-Kan model structure:} Weak equivalences are dg functors that are quasi-equivalences, and any dg category is a fibrant object.
		
		\item {\em Quasi-equiconic model structure:} Weak equivalences are {\em pretriangulated equivalences}, which are dg functors $\cC\to\cD$ that induce a quasi-equivalence $\Tw\,\cC\to\Tw\,\cD$, and fibrant objects are pretriangulated dg categories.
		
		\item {\em Morita model structure:} Weak equivalences are {\em Morita equivalences}, which are dg functors $\cC\to\cD$ that induce a quasi-equivalence $\Perf\,\cC\to\Perf\,\cD$, and fibrant objects are idempotent complete pretriangulated dg categories.
	\end{enumerate}
	All three model structures have the same cofibrations, which are retracts of semifree extensions. Consequently, they also share the same cofibrant objects, which are retracts of semifree dg categories.
\end{thm}

\begin{rmk}
	Any quasi-equivalence is a pretriangulated equivalence, and any pretriangulated equivalence is a Morita equivalence.
\end{rmk}

It is known that any morphism $\cC\to\cD$ in the homotopy category $\Ho(\cM)$ of a model category $\cM$ can be seen as a chain of objects and morphisms in $\cM$
\[\cC\xleftarrow{\sim}\cC'\to\cD'\xleftarrow{\sim}\cD\]
for some cofibrant object $\cC'$ and fibrant object $\cD'$, and arrows $\xleftarrow{\sim}$ are weak equivalences. Consequently, we can characterize the morphisms in $\Ho(\dgCat)$ through the following proposition:

\begin{prp}\label{prp:dgCat-morphisms}
	For given dg categories $\cC$ and $\cD$, a morphism $\cC\to\cD$ in $\Ho(\dgCat)$ can be characterized as a chain of dg categories and dg functors in the following ways:
	\begin{enumerate}
		\item $\cC\xleftarrow{\sim}\cC'\to\cD$, if $\dgCat$ is equipped with the Dwyer-Kan model structure.
		
		\item $\cC\xleftarrow{\sim}\cC'\to\Tw\,\cD$, if $\dgCat$ is equipped with the quasi-equiconic model structure.
		
		\item $\cC\xleftarrow{\sim}\cC'\to\Perf\,\cD$, if $\dgCat$ is equipped with the Morita model structure.
	\end{enumerate}
	Here, $\cC'$ is a cofibrant dg category, and $\xleftarrow{\sim}$ is a weak equivalence in the corresponding model structure.
\end{prp}

\begin{rmk}
	If $\cC$ is a cofibrant dg category, then we can replace each $\cC\xleftarrow{\sim}\cC'$ with $\cC$ in Proposition \ref{prp:dgCat-morphisms}.
\end{rmk}

Next, we introduce three distinct types of equivalency between dg categories, characterized by becoming isomorphic in the corresponding homotopy category $\Ho(\dgCat)$:

\begin{dfn} Let $\cC$ and $\cD$ be dg categories.
	\begin{enumerate}
		\item $\cC$ and $\cD$ are {\em quasi-equivalent} if there is a chain of dg categories and dg functors
		\[\cC\xleftarrow{\sim}\cC'\xrightarrow{\sim}\cD\]
		for some dg category $\cC'$, where each dg functor in the chain is a quasi-equivalence.
		
		\item $\cC$ and $\cD$ are {\em pretriangulated equivalent} if $\Tw\,\cC$ and $\Tw\,\cD$ are quasi-equivalent.
		
		\item $\cC$ and $\cD$ are {\em Morita equivalent} if $\Perf\,\cC$ and $\Perf\,\cD$ are quasi-equivalent.
	\end{enumerate}
\end{dfn}

As a result, we have two distinct types of generations for a dg category, defined as follows:

\begin{dfn} Let $\cC$ be a dg category. Let $\{L_i\}$ be a collection of objects in $\cC$. We say
	\begin{enumerate}		
		\item $\{L_i\}$ {\em generates} $\cC$ if the full dg subcategory of $\cC$ with the objects $\{L_i\}$ is pretriangulated equivalent to $\cC$,
		
		\item $\{L_i\}$ {\em split-generates} $\cC$ if the full dg subcategory of $\cC$ with the objects $\{L_i\}$ is Morita equivalent to $\cC$.
	\end{enumerate}
\end{dfn}

When $\cC$ is a dg category, and $S$ is a subset of closed degree zero morphisms in $\cC$, there exists a dg category $\cC[S^{-1}]$, known as the {\em dg localization} of $\cC$ at the morphisms in $S$. This localization is essentially obtained from $\cC$ by inverting morphisms in $S$. For a precise definition, one can refer to sources such as \cite{toen} or \cite{hocolim}. The dg localization is unique up to quasi-equivalence, and its existence is established in \cite{toen-morita}.

In the case where $\cC$ is a semifree dg category, we can explicitly describe $\cC[S^{-1}]$:

\begin{prp}[\cite{hocolim}]\label{prp:localise-semifree}
	When $\cC$ is a semifree dg category, and $S = \{g_i \colon A_i \to B_i\}$ is a subset of closed degree zero morphisms in $\cC$, the dg localization $\cC[S^{-1}]$ can be viewed as the semifree extension of $\cC$ by the morphisms $g'_i, \hat{g}_i, \check{g}_i, \bar{g}_i$
		\[\begin{tikzcd}
			A_i\ar[loop left,"\hat g_i"]\ar[r,"g_i"]\ar[r,"\bar g_i",bend left=60] & B_i\ar[l,"g'_i",bend left=30]\ar[loop right,"\check g_i"]
		\end{tikzcd}\]
		for each $i$, with the gradings
		\[|g_i'|=0,\qquad |\hat g_i|=|\check g_i|=-1,\qquad |\bar g_i|=-2,\]
		and with the differentials
		\[dg_i'=0,\qquad d\hat g_i=1_{A_i}-g_i'g_i,\qquad d\check g_i=1_{B_i}-g_ig_i',\qquad d\bar g_i=g_i\hat g_i - \check g_ig_i .\]
\end{prp}

Next, we explore the colimit of diagrams of dg categories:

\begin{prp}[\cite{hocolim}]\label{prp:colimit-dg}
	Let $G\colon \cA\to \cC$ be a dg functor, and $F\colon \cA \to \cB$ be a semifree extension by a set of objects $R$ and a set of morphisms $S$. Then, there exists a pushout (colimit) square
		\[\begin{tikzcd}
			\cB\ar[r,"\bar G"] & \cD \\
			\cA\ar[u,"F"]\arrow[r, "G"] & \cC \arrow[u, "\bar F"]
		\end{tikzcd}\]
		where $\bar F\colon\cC\to\cD$ is a semifree extension by the set of objects $R$ and a set of morphisms
		\[\bar S:=\{\bar f\colon \bar G(A)\to \bar G(B)\vb f\colon A\to B\text{ in }S\}\]
		with $|\bar f|:=|f|$ and $d\bar f:=\bar G(df)$, and
		\[\bar G(A):=\begin{cases}
			G(A) &\text{if } A\in\Ob\,\cA\\
			A &\text{if } A\in R
		\end{cases},\qquad\text{and}\qquad
		\bar G(f):=\begin{cases}
			G(f) &\text{if } f\in\mathrm{Mor}\,\cA\\
			\bar f &\text{if } f\in S
		\end{cases} .\]
\end{prp}

\begin{rmk}
	In a more casual sense, the colimit $\cD$ in Proposition \ref{prp:colimit-dg} can be thought as $\cB\amalg\cC$ after the identification of the images of $F$ with the images of $G$.
\end{rmk}

\begin{rmk}
	Given a semifree dg category $\cC$ with a set of generating morphisms $\{f_i\}$, consider a morphism $f\colon A\to B$ in $\{f_i\}$ where $A\neq B$. The dg localization $\cC[f^{-1}]$ can be given by $\cC$ with the identifications $A=B$ and $f=1_{A=B}$. This description relies on a description of dg localization through a colimit diagram, as presented in \cite{toen-morita}, and Proposition \ref{prp:colimit-dg}.
\end{rmk}

Finally, we present two propositions from \cite{chekanov} and \cite{subcritical}, which can be thought as ``basis change'' and ``cancellation of generators'' for the morphisms of semifree dg categories, respectively. They are useful when we simplify a given semifree dg category.

\begin{prp}\label{prp:variable-change}
	Let $\cC$ be a semifree dg category with a set of generating morphisms $\{f_i\}$ (indexed by an ordinal). Define the morphisms
	\[\tilde f_i:=u_i f_i + g_i\]
	where $u_i$ is a unit in the coefficient ring $k$, and $g_i$ is a morphism in $\cC$ generated by the set $\{f_j\vb j<i\}$. Then, the set $\{\tilde f_i\}$ also generates the morphisms in $\cC$ semifreely.
\end{prp}

\begin{prp}\label{prp:destabilisation}
	Let $\cC$ be a semifree dg category, and $\cD$ be the semifree extension of $\cC$ by the morphisms $\{a_i,b_i\}$ such that $da_i=b_i$ for all $i$. Then, $\cC$ and $\cD$ are quasi-equivalent. 
\end{prp}

In the setting of Proposition \ref{prp:destabilisation}, we say $\cD$ is obtained from $\cC$ by {\em stabilization}, and $\cC$ is obtained from $\cD$ by {\em destabilization}.

\section{Homotopy colimit functor using model structures}\label{subsec:hocolim-via-model}

This section is a review of the homotopy colimit functor and related model-theoretical results. Our main reference is \cite{dwyer-spalinski}.

For a given model category $\cM$ (more generally, a category $\cM$ with weak equivalences), we write $\Ho(\cM)$ for its {\em homotopy category}, which is the category
\begin{itemize}
	\item whose objects are the same as the objects of $\cM$, and
	\item whose morphisms are generated by the morphisms of $\cM$ and the formal inverses of the weak equivalences.
\end{itemize}
It comes with the {\em localization functor}
\[l\colon\cM\to\Ho(\cM)\]
which is the identity on objects, and sending morphisms to themselves.

From now on, let $J$ be the category given by
\[a\leftarrow c\rightarrow b\]
where $a,b,c$ are the objects and the arrows are the morphisms.

\begin{dfn}
	Let $\cM$ be a category with weak equivalences, and $\cM^J$ be the category of functors $J\to\cM$ ($J$-diagrams in $\cM$) whose weak equivalences are the objectwise weak equivalences. The {\em homotopy colimit functor}
	\[\hocolim\colon\Ho(\cM^J)\to\Ho(\cM)\]
	is defined (up to natural equivalence) as the total left derived functor of the colimit functor
	\[\colim\colon\cM^J\to\cM .\]
\end{dfn}

If $\cM$ has a model structure, we have a more concrete way to express the homotopy colimit functor. To describe it, we consider an induced model structure on $\cM^J$ from the model structure on $\cM$:

\begin{prp}[\cite{hirschhorn},\cite{dugger}]\label{prp:reedy}
	Let $\cM$ be a model category.
	\begin{enumerate}
		\item $\cM^J$ has a model structure, called a {\em Reedy model structure}, whose cofibrant objects are the diagrams of the form
		\[\cA\xleftarrow{\alpha} \cC\xrightarrow{\beta} \cB\]
		where $\cA,\cB,\cC$ are cofibrant objects in $\cM$, and $\beta$ is a cofibration.
		
		\item If $\cM^J$ is equipped with the Reedy model structure above, then $\colim\colon\cM^J\to\cM$ preserves cofibrations and acyclic cofibrations.
	\end{enumerate}
\end{prp}

Before describing the homotopy colimit functor using model structures, we need to define cofibrant resolution functors:

\begin{prp}[\cite{dwyer-spalinski}]\label{prp:cofibrant-resolution}	
	Let $\cM$ be a model category. For any $X\in\cM$, there exists a cofibrant object $Q(X)\in\cM$ and an acyclic fibration $p_X\colon Q(X)\to X$ by model category axioms. Then for any morphism $f\colon X\to Y$, there exists a unique morphism $\tilde f\colon Q(X)\to Q(Y)$ up to right homotopy such that the diagram
	\begin{equation*}\label{eq:cofibrant-resolution}\begin{tikzcd}
			Q(X)\ar[r,"\tilde f"]\ar[d,"p_X"] & Q(Y)\ar[d,"p_{Y}"]\\
			X\ar[r,"f"] & Y
	\end{tikzcd}\end{equation*}
	commutes.
\end{prp}

\begin{dfn}\label{dfn:cof-res-functor}
	Let $\cM$ be a model category.
	\begin{enumerate}
		\item We define $\cM_c$ as the full subcategory of $\cM$ consisting of cofibrant objects, and $\pi\cM_c$ as the category with the same objects as $\cM_c$ whose morphisms are right homotopy classes of morphisms in $\cM_c$.
		
		\item For any $X\in\cM$, there exists a cofibrant object $Q(X)\in\cM$ and an acyclic fibration $p_X\colon Q(X)\to X$ by model category axioms. Then we define a {\em cofibrant resolution functor} as
		\begin{align*}
			Q\colon \cM &\to\pi\cM_c\\
			X &\mapsto Q(X) && \text{on objects}\\
			f &\mapsto [\tilde f] &&\text{on morphisms}
		\end{align*}
		where given $f\colon X\to Y$, the morphism $\tilde f$ is determined by Proposition \ref{prp:cofibrant-resolution}, and $[\tilde f]$ is the right homotopy class of $\tilde f$. The well-definedness of this functor follows from Proposition \ref{prp:cofibrant-resolution}.
	\end{enumerate}
\end{dfn}
 
Now, we present the alternative description of the homotopy colimit functor:

\begin{thm}[\cite{dwyer-spalinski}]\label{thm:dwyer-hocolim}
	Let $\cM$ be a model category, and equip $\cM^J$ with the Reedy model structure induced by $\cM$ as in Proposition \ref{prp:reedy}. The unique lift of the composition
	\[\cM^J\xrightarrow{Q}\pi(\cM^J)_c\xrightarrow{l\,\circ\,\colim}\Ho(\cM)\]
	gives the homotopy colimit functor
	\[\hocolim\colon\Ho(\cM^J)\rightarrow\Ho(\cM)\]
	where $Q$ is a cofibrant resolution functor.
\end{thm}

\begin{rmk}\label{rmk:homotopy-lift}\mbox{}
	\begin{enumerate}
		\item Although the colimit functor is not well-defined on $\pi(\cM^J)_c$, the functor $l\circ\colim$ above is well-defined since it identifies right homotopic maps between cofibrant objects. This follows from the second item in Proposition \ref{prp:reedy}. See \cite{dwyer-spalinski} for the details.
		\item By the lift, we mean that the triangle
		\[\begin{tikzcd}[column sep=2cm]
			\cM^J\ar[r,"l\,\circ\,\colim\,\circ\, Q"]\ar[d,"l"] & \Ho(\cM)\\
			\Ho(\cM^J)\ar[ru,"\hocolim"']
		\end{tikzcd}\]
		commutes. The existence and uniqueness follow from the fact that $l\circ\colim\circ Q$ sends weak equivalences to isomorphisms, which again follows from the second item in Proposition \ref{prp:reedy}.
	\end{enumerate}
\end{rmk}

Before presenting a corollary of Theorem \ref{thm:dwyer-hocolim}, we describe a functor transforming a given diagram to a more manageable one:

\begin{prp}\label{prp:diagram-diagonal}\mbox{}
	\begin{enumerate}
		\item There exists a functor $T\colon\cM^J\to\cM^J$ such that
		\begin{itemize}
			\item $T$ sends an object of the form
			\[X:=(\cA\xleftarrow{\alpha}\cC\xrightarrow{\beta}\cB)\]
			to the object
			\[T(X):=(\cA\amalg \cB\xleftarrow{\alpha\amalg\beta}\cC\amalg \cC\xrightarrow{\nabla_{\cC}}\cC)\]
			where $\nabla_{\cC}$ is the codiagonal for $\cC$, and
		
			\item $T$ sends a morphism of the form
			\[\begin{tikzcd}
				X\ar[d,"F",blue]\\
				X'
			\end{tikzcd}
			\quad:=\quad
			\begin{tikzcd}
				\cA\ar[d,"F_{\cA}",blue] & \cC\ar[l,"\alpha"']\ar[r,"\beta"]\ar[d,"F_{\cC}",blue] & \cB\ar[d,"F_{\cB}",blue]\\
				\cA' & \cC'\ar[l,"\alpha'"']\ar[r,"\beta'"] & \cB'
			\end{tikzcd}\]
			to the morphism
			\[\begin{tikzcd}
				T(X)\ar[d,"T(F)",blue]\\
				T(X')
			\end{tikzcd}
			\quad:=\quad
			\begin{tikzcd}
				\cA\amalg \cB\ar[d,"F_{\cA}\amalg F_{\cB}",blue] & \cC\amalg \cC\ar[l,"\alpha\amalg\beta"']\ar[r,"\nabla_{\cC}"]\ar[d,"F_{\cC}\amalg F_{\cC}",blue] & \cC\ar[d,"F_{\cC}",blue]\\
				\cA'\amalg \cB' & \cC'\amalg \cC'\ar[l,"\alpha'\amalg\beta'"']\ar[r,"\nabla_{\cC'}"] & \cC'
			\end{tikzcd} .\]
		\end{itemize}
		\item The colimit functor satisfies
		\[\colim\circ T=\colim .\]
		\item $T$ sends weak equivalences to weak equivalences, hence it induces the functor
		\[\Ho(T)\colon\Ho(\cM^J)\to\Ho(\cM^J)\]
		satisfying
		\[\hocolim\circ\Ho(T)=\hocolim .\]
	\end{enumerate}
\end{prp}

\begin{proof}
	It is straightforward to check.
\end{proof}

Using Proposition \ref{prp:diagram-diagonal}, we have a slight improvement of Theorem \ref{thm:dwyer-hocolim} for constructing the homotopy colimit functor, which only requires us to construct a cofibrant resolution functor $Q$ for the image of the functor $T$. This will be useful in Section \ref{sec:hocolim-semifree}.

\begin{cor}\label{cor:dwyer-hocolim}
	Let $\cM$ be a model category, and equip $\cM^J$ with the Reedy model structure induced by $\cM$ as in Proposition \ref{prp:reedy}. The unique lift of the composition
	\[\cM^J\xrightarrow{T}\cM^J\xrightarrow{Q}\pi(\cM^J)_c\xrightarrow{l\,\circ\,\colim}\Ho(\cM)\]
	gives the homotopy colimit functor
	\[\hocolim\colon\Ho(\cM^J)\rightarrow\Ho(\cM)\]
	where $T$ is the functor described in Proposition \ref{prp:diagram-diagonal}, and $Q$ is a cofibrant resolution functor.
\end{cor}

\begin{proof}
	The statement follows from the commutativity of the diagram
	\[\begin{tikzcd}
		\cM^J\ar[r,"T"]\ar[d,"l"] & \cM^J\ar[d,"l"]\ar[r,"Q"] & \pi(\cM^J)_c\ar[d,"l\,\circ\,\colim"]\\
		\Ho(\cM^J)\ar[r,"\Ho(T)"]\ar[rr,"\hocolim", bend right=25] & \Ho(\cM^J)\ar[r,"\hocolim"] & \Ho(\cM)
	\end{tikzcd}\]
	by Theorem \ref{thm:dwyer-hocolim} and Proposition \ref{prp:diagram-diagonal}.
\end{proof}

\section{Cylinder functor for the category of dg categories}\label{sec:cylinder-functor}

Our goal is to give a formula for the homotopy colimit functor $\hocolim\colon \Ho(\cM^J)\to\Ho(\cM)$ for the case $\cM=\dgCat$, where $\dgCat$ is the model category of small $k$-linear dg categories for a commutative ring $k$ with Dwyer-Kan, quasi-equiconic, or Morita model structure. Their weak equivalences are quasi-equivalences, pretriangulated equivalences, and Morita equivalences, respectively (see Theorem \ref{thm:dgcat-model}). Our formula will work for all these model structures. To express the formula, first we need to describe a cofibrant resolution functor, which will be done in Section \ref{sec:hocolim-semifree}.

Here, we will define a cylinder functor for the semifree dg categories in $\dgCat$, in other words, we will give a functorial construction of cylinder objects for semifree dg categories. We will use this construction to describe a cofibrant resolution functor in Section \ref{sec:hocolim-semifree}. Also, in Section \ref{sec:cofibration-category}, the cylinder functor will play a pivotal role in defining an $I$-category and a cofibration category of semifree dg categories.

We first recall the definition of a cylinder functor:

\begin{dfn}\label{dfn:cylinder-functor}
	Let $\cM$ be a model category. A functor $I\colon\cM\to\cM$ is called a {\em cylinder functor}, if there exists a cofibration
	\[i_{\cC}\colon\cC\amalg\cC\to I(\cC)\]
	and an acyclic fibration
	\[p_{\cC}\colon I(\cC)\to\cC\]
	such that $p_{\cC}\circ i_{\cC}=\nabla_{\cC}$ (the codiagonal of $\cC$) for all $\cC\in\cM$ (in other words, $I(\cC)$ is a {\em cylinder object} for $\cC$), and the diagram
	\begin{equation}\label{eq:cylinder-functor}\begin{tikzcd}
		\cC\amalg\cC\ar[r,"F\amalg F"]\ar[d,"i_{\cC}"] & \cD\amalg\cD\ar[d,"i_{\cD}"]\\
		I(\cC)\ar[r,"I(F)"]\ar[d,"p_{\cC}"] & I(\cD)\ar[d,"p_{\cD}"]\\
		\cC\ar[r,"F"] & \cD
	\end{tikzcd}\end{equation}
	commutes for any morphism $F\colon \cC\to\cD$ in $\cM$.
\end{dfn}

From now on, we focus on $\cM=\dgCat$ with Dwyer-Kan model structure. Everything here still holds if we work with quasi-equiconic or Morita model structure.  Before defining a cylinder functor, recall that \cite{hocolim} defined a cylinder {\em object} $\Cyl(\cC)$ for any semifree dg category $\cC$:

\begin{dfn}[\cite{hocolim}]\label{dfn:cylinder-pre-dg}
	Let $\cC$ be a semifree dg category, and $i_1,i_2\colon\cC\to\cC\amalg\cC$ be the inclusions to the first and second copies, respectively. We define $\Cyl(\cC)$ as the semifree extension of $\cC\amalg\cC$ by the morphisms comprised of
	\begin{itemize}
		\item the morphisms $t_C,t_C',\hat t_C,\check t_C,\bar t_C$
		\[\begin{tikzcd}
			i_1(C)\ar[loop left,"\hat t_C"]\ar[r,"t_C"]\ar[r,"\bar t_C",bend left=60] & i_2(C)\ar[l,"t_C'",bend left=30]\ar[loop right,"\check t_C"]
		\end{tikzcd}\]
		for each $C\in\cC$, with the gradings
		\[|t_C|=|t_C'|=0,\qquad |\hat t_C|=|\check t_C|=-1,\qquad |\bar t_C|=-2,\]
		and with the differentials
		\[d t_C=d t_C'=0,\qquad d\hat t_C=1_{i_1(C)}-t_C'\circ t_C,\qquad d\check t_C=1_{i_2(C)}-t_C\circ t_C',\qquad d\bar t_C=t_C\circ\hat t_C - \check t_C\circ t_C ,\]
		\item degree $|f|-1$ morphism $t_f\colon i_1(A)\to i_2(B)$ for each generating morphism $f\in\hom^*_{\cC}(A,B)$, with the differential
		\[dt_f=(-1)^{|f|}(i_2(f)\circ t_A - t_B\circ i_1(f))+\text{correction term}\]
		where the correction term is $0$ if $df=0$.
	\end{itemize}
	 If $df\neq 0$, the correction term is given as follows: If 
	\[df = c 1_A + \sum_{i=1}^m c_i f_{i,n_i}\circ \ldots \circ f_{i,j}\circ\ldots\circ f_{i,1}\]
	where $f_{i,j}$ are generating morphisms of $\cC$, and $c, c_i\in k$ ($c=0$ if $A\neq B$), then
	\[\text{correction term}=\sum_{i=1}^m c_i \sum_{j=1}^{n_i}(-1)^{|f_{i,j-1}|+\ldots+|f_{i,1}|} i_2(f_{i,n_i}) \ldots i_2(f_{i,j+1}) \circ t_{f_{i,j}}\circ  i_1(f_{i,j-1})\ldots i_1(f_{i,1}).\]
\end{dfn}

\begin{rmk}\label{rmk:cyl-object-well-defined}
	The semifree dg category $\Cyl(\cC)$, associated with a given semifree dg category $\cC$, is well-defined up to isomorphism. In other words, it does not depend on the choice of generating morphisms of $\cC$. This can be verified directly or by considering Remark \ref{rmk:cyl-functor-well-defined}.
\end{rmk}

\begin{rmk}\label{rmk:cylinder-dg-loc}
	By the description of the dg localization in Proposition \ref{prp:localise-semifree}, $\Cyl(\cC)$ is the dg localization $\Cyl_0(\cC)[S^{-1}]$ (up to quasi-equivalence), where $\Cyl_0(\cC)$ is the semifree extension of $\cC\amalg\cC$ by the morphisms comprised of
	\begin{itemize}
		\item closed degree zero morphism $t_C\colon i_1(C)\to i_2(C)$ for each $C\in\cC$,
		\item degree $|f|-1$ morphism $t_f\colon i_1(A)\to i_2(B)$ for each generating morphism $f\in\hom^*_{\cC}(A,B)$, with the differential given in Definition \ref{dfn:cylinder-pre-dg},
	\end{itemize}
	and $S=\{t_C\vb C\in\cC\}$.
\end{rmk}

\begin{thm}[\cite{hocolim}]\label{thm:cylinder-object}
	Let $\cC$ be a semifree dg category. $\Cyl(\cC)$ defined in Definition \ref{dfn:cylinder-pre-dg} is a cylinder object for $\cC$. That is, the semifree extension
	\[i_{\cC}:=i_1\amalg i_2\colon\cC\amalg\cC\to\Cyl(\cC)\]
	is a cofibration, and the functor
	\begin{align*}
		p_{\cC}\colon\Cyl(\cC)&\to\cC\\
		i_1(C),i_2(C)&\mapsto C,\quad t_C,t_C'\mapsto 1_C,\quad \hat t_C,\check t_C,\bar t_C\mapsto 0 &&\text{for each object $C\in\cC$}\\
		i_1(f),i_2(f)&\mapsto f,\quad t_f\mapsto 0 &&\text{for each generating morphism $f$ in $\cC$}
	\end{align*}
	is an acyclic fibration, and they satisfy $p_{\cC}\circ i_{\cC}=\nabla_{\cC}$.
\end{thm}

So, Theorem \ref{thm:cylinder-object} builds the cylinder functor on objects of $\dgCat$, but not on morphisms. The functoriality of the cylinder object construction is not discussed in \cite{hocolim}. Here, we will make the construction functorial. This is the main goal of this section. First, we need some definitions and properties.

Given a semifree dg category $\cC$, Definition \ref{dfn:cylinder-pre-dg} defines $\Cyl(\cC)$ and for each generating morphism $f$ of $\cC$, it specifies a generating morphism $t_f$ of $\Cyl(\cC)$. Here, we extend this and define a morphism $t_{\theta}$ in $\Cyl(\cC)$ for each morphism $\theta$ in $\cC$, which will be useful later:

\begin{dfn}\label{dfn:generalised-t-morphisms}
	Let $\cC$ be a semifree dg category. Let $\theta\in\hom^*_{\cC}(A,B)$ be given by
	\[\theta=c 1_A + \sum_{i=1}^m c_i f_{i,n_i}\circ\ldots\circ f_{i,j}\circ\ldots\circ f_{i,1}\]
	for some generating morphisms $f_{i,j}$ of $\cC$ and $c, c_i\in k$ ($c=0$ if $A\neq B$). We define degree $|\theta|-1$ morphism $t_{\theta}\in\hom^*_{\Cyl(\cC)}(i_1(A),i_2(B))$ by
	\[t_{\theta}:=\sum_{i=1}^m c_i \sum_{j=1}^{n_i}(-1)^{|f_{i,j-1}|+\ldots+|f_{i,1}|} i_2(f_{i,n_i})\circ\ldots\circ i_2(f_{i,j+1})\circ t_{f_{i,j}}\circ i_1(f_{i,j-1})\circ\ldots\circ i_1(f_{i,1}) .\]
\end{dfn}

\begin{prp}\label{prp:properties-t-theta}
	Let $\cC$ be a semifree dg category, and let $\theta\in\hom^*_{\cC}(A,B)$. We have the identity
	\begin{equation}\label{eq:theta-differential}
		dt_{\theta}=(-1)^{|\theta|}(i_2(\theta)\circ t_A - t_B\circ i_1(\theta)) + t_{d\theta} .
	\end{equation}
	Moreover, if
	\begin{equation}\label{eq:theta-via-theta}
		\theta=c 1_A + \sum_{i=1}^m c_i \theta_{i,n_i}\circ\ldots\circ \theta_{i,j}\circ\ldots\circ \theta_{i,1}
	\end{equation}
	for some morphisms $\theta_{i,j}$ in $\cC$ (not necessarily generating morphisms) and $c, c_i\in k$ ($c=0$ if $A\neq B$), we have the identity
	\begin{equation}\label{eq:theta-equation}
		t_{\theta}=\sum_{i=1}^m c_i \sum_{j=1}^{n_i}(-1)^{|\theta_{i,j-1}|+\ldots+|\theta_{i,1}|} i_2(\theta_{i,n_i})\circ\ldots\circ i_2(\theta_{i,j+1})\circ t_{\theta_{i,j}}\circ i_1(\theta_{i,j-1})\circ\ldots\circ i_1(\theta_{i,1}) .
	\end{equation}
\end{prp}

\begin{proof}
	First, we remark that \eqref{eq:theta-differential} holds by definition when $\theta$ is a generating morphism. Now, assume $\theta$ is as in Definition \ref{dfn:generalised-t-morphisms}. Then the assignment $\theta\mapsto t_{\theta}$ is linear. Hence, it is enough to prove \eqref{eq:theta-differential} for $\theta=1_A$ and $\theta=f_n\circ\ldots\circ f_1$ for any generating morphisms $f_i$ of $\cC$. This is straightforward to check using the remark in the beginning of the proof.
	
	Now, assume $\theta$ is given as in \eqref{eq:theta-via-theta}. By the linearity of the assignment $\theta\mapsto t_{\theta}$, we only need to prove \eqref{eq:theta-equation} for $\theta=\theta_n\circ\ldots\circ\theta_1$ for any given morphisms $\theta_i$. Furthermore, again by the linearity, we can assume that each $\theta_i$ is given as a product of generating morphisms $f_j$ of $\cC$. Then, it is straightforward to check that \eqref{eq:theta-equation} holds.
\end{proof}

We are now ready to state one of the main results of this section:

\begin{thm}\label{thm:cylinder-morphism}
	Let $F\colon\cC\to\cD$ be a dg functor between semifree dg categories.
	\begin{enumerate}
		\item There is a dg functor $\Cyl(F)\colon\Cyl(\cC)\to\Cyl(\cD)$ that is an extension of the dg functor
		\[F\amalg F\colon \cC\amalg\cC \to \cD\amalg\cD\]
		i.e.
		\begin{align*}
			\Cyl(F)\colon\Cyl(\cC)&\to\Cyl(\cD)\\
			i_1(C),i_2(C)&\mapsto i_1(F(C)),i_2(F(C)) &&\text{respectively, for each $C\in\cC$}\\
			i_1(\theta),i_2(\theta)&\mapsto i_1(F(\theta)),i_2(F(\theta)) &&\text{respectively, for each morphism $\theta$ in $\cC$}
		\end{align*}
		by additionally specifying
		\begin{align*}
			t_C, t_C',\hat t_C,\check t_C,\bar t_C&\mapsto t_{F(C)}, t_{F(C)}',\hat t_{F(C)},\check t_{F(C)},\bar t_{F(C)}&&\text{respectively, for each object $C\in\cC$}\\
			t_f&\mapsto t_{F(f)} &&\text{for each generating morphism $f$ in $\cC$}
		\end{align*}
		where $t_{F(f)}$ is defined according to Definition \ref{dfn:generalised-t-morphisms}.
		
		\item The dg functor $\Cyl(F)$ satisfies
		\begin{equation}\label{eq:cyl-theta}\Cyl(F)(t_{\theta})=t_{F(\theta)}\end{equation}
		for any morphism $\theta$ in $\cC$.
		
		\item The diagram
		\begin{equation}\label{eq:cylinder-morphism-diagram}\begin{tikzcd}
				\cC\amalg\cC\ar[r,"F\amalg F"]\ar[d,"i_{\cC}"] & \cD\amalg\cD\ar[d,"i_{\cD}"]\\
				\Cyl(\cC) \ar[r,"\Cyl(F)"]\ar[d,"p_{\cC}"] & \Cyl(\cD)\ar[d,"p_{\cD}"]\\
				\cC\ar[r,"F"] & \cD
		\end{tikzcd}\end{equation}
		commutes.
	\end{enumerate} 
\end{thm}

\begin{proof}
	We need to show that $\Cyl(F)$ is indeed a dg functor (it is already a functor by definition). The only nontrivial part is showing that $d(\Cyl(F)(t_f))=\Cyl(F)(dt_f)$ for any generating morphism $f$ in $\cC$. To verify this, we need to prove the identity
	\[\Cyl(F)(t_{\theta})=t_{F(\theta)}\]
	for any morphism $\theta$ in $\cC$. Assume without loss of generality that $\theta$ is given by
	\[\theta= f_n\circ\ldots\circ f_1\]
	for some generating morphisms $f_j$ of $\cC$. Then, since $\Cyl(F)$ is a functor, we have
	\begin{align*}
		\Cyl(F)(t_{\theta})&=\Cyl(F)\left(\sum_{j=1}^{n}(-1)^{|f_{j-1}|+\ldots+|f_{1}|} i_2(f_{n})\circ\ldots\circ i_2(f_{j+1})\circ t_{f_{j}}\circ i_1(f_{j-1})\circ\ldots\circ i_1(f_{1})\right)\\
		&=\sum_{j=1}^{n}(-1)^{|f_{j-1}|+\ldots+|f_{1}|} i_2(F(f_{n}))\circ\ldots\circ i_2(F(f_{j+1}))\circ t_{F(f_{j})}\circ i_1(F(f_{j-1}))\circ\ldots\circ i_1(F(f_{1}))\\
		&=t_{F(\theta)}
	\end{align*}
	by the identity (\ref{eq:theta-equation}) since $F(\theta)=F(f_n)\circ\ldots\circ F(f_1)$.
	
	Using this and the identity \eqref{eq:theta-differential}, for any generating morphism $f\colon A\to B$ in $\cC$, we see that
	\begin{align*}
		d(\Cyl(F)(t_{f}))&=dt_{F(f)}\\
		&=(-1)^{|F(f)|}(i_2(F(f))\circ t_{F(A)} - t_{F(B)}\circ i_1(F(f))) + t_{d(F(f))}\\
		&=(-1)^{|f|}(i_2(F(f))\circ t_{F(A)}  - t_{F(B)}\circ i_1(F(f))) + t_{F(df)}\\
		&=\Cyl(F)((-1)^{|f|}(i_2(f)\circ t_A - t_B\circ i_1(f)) + t_{df})\\
		&=\Cyl(F)(dt_{f})
	\end{align*}
	which shows that $\Cyl(F)$ is a dg functor.
	
	Finally, the commutation of the diagram (\ref{eq:cylinder-morphism-diagram}) is obvious from the construction.
\end{proof}

Theorem \ref{thm:cylinder-morphism} suffices for the purpose of defining the homotopy colimit functor in Section \ref{sec:hocolim-semifree}. However, we proceed to establish a cylinder functor for semifree dg categories in Corollary \ref{cor:cyl-functor}. This construction plays a key role in defining an $I$-category (a category with a natural cylinder functor), thus giving rise to a cofibration category of semifree dg categories in Section \ref{sec:cofibration-category}.

\begin{cor}\label{cor:cyl-functor}
	The assignment
	\begin{align*}
		\Cyl\colon\dgCat_s &\to\dgCat_s\\
		\cC&\mapsto \Cyl(\cC)\\
		F &\mapsto \Cyl(F)
	\end{align*}
	is a cylinder functor for $\dgCat_s$, where $\dgCat_s$ is the full subcategory of the model category of dg categories $\dgCat$ (with Dwyer-Kan, quasi-equiconic, or Morita model structure) consisting of semifree dg categories, and $\Cyl(\cC)$ is as defined in Definition \ref{dfn:cylinder-pre-dg}, and $\Cyl(F)$ is as defined in Theorem \ref{thm:cylinder-morphism}.
\end{cor}

\begin{proof}
	Clearly, $\Cyl(1_{\cC})=1_{\Cyl(\cC)}$ where $1_{\cC}\colon\cC\to\cC$ is the identity functor on $\cC$. Also, the equality $\Cyl(G\circ F)=\Cyl(G)\circ\Cyl(F)$ directly follows from the definition of $\Cyl(F)$ and the identity (\ref{eq:cyl-theta}), hence $\Cyl$ is a functor. Then, the commutative diagram \eqref{eq:cylinder-morphism-diagram} shows that $\Cyl$ is indeed a cylinder functor in the sense of Definition \ref{dfn:cylinder-functor}.
\end{proof}

\begin{rmk}\label{rmk:cyl-functor-well-defined}
	To establish the well-definedness of the functor $\Cyl$ up to natural isomorphism, let us denote the application of $\Cyl$ to a semifree dg category $\cC$ with a predetermined set of generating morphisms $\{f_i\}$ by $\Cyl(\cC,\{f_i\})$. By applying $\Cyl$ to the identity functor $1_{\cC} \colon\cC\to\cC$, we get an isomorphism
	\[\Cyl(1_{\cC})\colon\Cyl(\cC,\{f_i\})\xrightarrow{\sim}\Cyl(\cC,\{g_i\})\]
	for any set of generating morphisms $\{g_i\}$ of $\cC$, which follows from the fact that $\Cyl$ is a functor. Hence, $\Cyl(\cC)$ is well-defined up to isomorphism. Lastly, if there is a functor $F\colon \cC\to\cD$ between semifree dg categories, it induces a commutative square
	\[\begin{tikzcd}
		\Cyl(\cC,\{f_i\})\ar[r,"\Cyl(1_{\cC})","\sim"']\ar[d,"\Cyl(F)"] & \Cyl(\cC,\{g_i\})\ar[d,"\Cyl(F)"]\\
		\Cyl(\cD,\{f_i'\})\ar[r,"\Cyl(1_{\cD})","\sim"'] & \Cyl(\cD,\{g'_i\})
	\end{tikzcd} .\]
	Here, $\{f_i'\}$ and $\{g_i'\}$ are two arbitrary sets of generating morphisms of $\cD$. Hence, $\Cyl$ is well-defined up to natural isomorphism.
\end{rmk}

\begin{rmk}\label{rmk:cylinder-functor-dg-algebras}
	The cylinder object and functor can be also defined for the model category of dg algebras. In that case, for a given semifree dg algebra $\cC$, $\cC\amalg\cC$ is the semifree dg algebra (with the same unique object as $\cC$ by definition) whose generating morphisms are doubled, and $i_1,i_2\colon\cC\to\cC\amalg\cC$ are the obvious inclusions. $\Cyl(\cC)$ is defined as the semifree extension of $\cC\amalg\cC$ by the morphisms $t_f$, where $f$ is a generating morphism of $\cC$ (no $t_C,t'_C,\hat t_C,\check t_C,\bar t_C$ are used for $C\in\cC$). All the formulas still hold by setting
	\[t_C=1,\quad t_C'=1,\quad \hat t_C=0,\quad\check t_C=0,\quad \bar t_C=0.\]
	Given a morphism $F\colon\cC\to\cD$ between semifree dg algebras, $\Cyl(F)$ is just defined by
	\[\Cyl(F)(t_f)=t_{F(f)}\]
	for every generating morphism $f$ in $\cC$. Hence, this defines a cylinder functor for the model category of dg algebras as in Corollary \ref{cor:cyl-functor}. The proofs are the similar in the case of semifree dg algebras.
\end{rmk}

Finally, we want to discuss the cylinder object for the dg localization $\cC[S^{-1}]$, where $\cC$ is a semifree dg category and $S$ is a subset of closed degree zero morphisms in $\cC$. We note that since $\cC[S^{-1}]$ can be also seen as a semifree dg category, one can apply Theorem \ref{thm:cylinder-object} and \ref{thm:cylinder-morphism} to $\cC[S^{-1}]$. But we would like to discuss a simpler cylinder object for $\cC[S^{-1}]$ for later convenience.

By the description of dg localization in Proposition \ref{prp:localise-semifree}, we know that $\cC[S^{-1}]$ can be express as a semifree extension of $\cC$ by the morphisms
\[g'\colon B\to A,\quad \hat g\colon A\to A,\quad \check g\colon B\to B,\quad\bar g\colon A\to B\]
for every $g\colon A\to B$ in $S$, whose gradings and differentials are given as in Proposition \ref{prp:localise-semifree}. Hence, the generating morphisms of $\cC[S^{-1}]$ are the generating morphisms of $\cC$ and the morphisms $g',\hat g,\check g,\bar g$ for each morphism $g$ in $S$.

Recall that $i_1,i_2\colon\cC\to\cC\amalg\cC\hookrightarrow\Cyl(\cC)$ are the inclusions to the first and second copies, respectively. Then, it is easy to see that $\Cyl(\cC[S^{-1}])$ is the semifree extension of $\Cyl(\cC)[(i_1(S)\sqcup i_2(S))^{-1}]$ by the morphisms
\[t_{g'}\colon i_1(B)\to i_2(A),\quad t_{\hat g}\colon i_1(A)\to i_2(A),\quad t_{\check g}\colon i_1(B)\to i_2(B),\quad t_{\bar g}\colon i_1(A)\to i_2(B)\]
for every $g\colon A\to B$ in $S$.

Note that $\Cyl(\cC)[(i_1(S)\sqcup i_2(S))^{-1}]$ is the semifree extension of $\cC[S^{-1}]\amalg\cC[S^{-1}]$ by the morphisms
\begin{itemize}
	\item $t_C,t_C',\hat t_C,\check t_C,\bar t_C$ for each $C\in\cC$, and
	\item $t_f$ for each generating morphism $f$ in $\cC$.
\end{itemize}
In \cite{hocolim}, it is shown that one can choose a cylinder object for $\cC[S^{-1}]$ that is simpler than $\Cyl(\cC[S^{-1}])$ in the sense that it does not need the generating morphisms $t_{g'},t_{\hat g},t_{\check g},t_{\bar g}$:

\begin{thm}[\cite{hocolim}]\label{thm:cylinder-loc}
	Let $\cC$ be a semifree dg category, and $S$ be a subset of closed degree zero morphisms in $\cC$. Then $\Cyl(\cC)[(i_1(S)\sqcup i_2(S))^{-1}]$ is a cylinder object for $\cC[S^{-1}]$. That is, the semifree extension
	\[i_{\cC[S^{-1}]}:=i_1\amalg i_2\colon\cC[S^{-1}]\amalg\cC[S^{-1}]\to\Cyl(\cC)[(i_1(S)\sqcup i_2(S))^{-1}]\]
	is a cofibration, and the functor
	\begin{align*}
		p_{\cC[S^{-1}]}\colon\Cyl(\cC)[(i_1(S)\sqcup i_2(S))^{-1}]&\to\cC[S^{-1}]\\
		i_1(C),i_2(C)\mapsto C,\quad t_C,t_C'&\mapsto 1_C,\quad\hat t_C,\check t_C,\bar t_C\mapsto 0 &&\text{for each object $C\in\cC$}\\
		t_f&\mapsto 0 &&\text{for each generating morphism $f$ in $\cC$}\\
		i_1(\theta),i_2(\theta)&\mapsto \theta&&\text{for each morphism $\theta$ in $\cC[S^{-1}]$}
	\end{align*}
	is an acyclic fibration, and they satisfy $p_{\cC[S^{-1}]}\circ i_{\cC[S^{-1}]}=\nabla_{\cC[S^{-1}]}$.
\end{thm}

Let $\cC$ and $\cD$ be a semifree dg categories, and $S_{\cC}$ and $S_{\cD}$ be subsets of closed degree zero morphisms in $\cC$ and $\cD$, respectively. For a given functor $F\colon\cC[S_{\cC}^{-1}]\to\cD[S_{\cD}^{-1}]$, Theorem \ref{thm:cylinder-morphism} constructs the dg functor
\[\Cyl(F)\colon\Cyl(\cC[S_{\cC}^{-1}])\to\Cyl(\cD[S_{\cD}^{-1}]) .\]
However, in such cases, i.e., when a dg category is given as a dg localization, we want to work with its simpler cylinder object given by Theorem \ref{thm:cylinder-loc}. Hence, we conclude this section with a construction of a dg functor between the simpler cylinder objects:

\begin{thm}\label{thm:cylinder-morphism-loc}
	Let $\cC$ and $\cD$ be a semifree dg categories, and $S_{\cC}$ and $S_{\cD}$ be subsets of closed degree zero morphisms in $\cC$ and $\cD$, respectively. Let $F\colon\cC[S_{\cC}^{-1}]\to\cD[S_{\cD}^{-1}]$ be a dg functor.
	\begin{enumerate}
	\item There is a dg functor
	\[\Cyl(F)\colon\Cyl(\cC)[(i_1(S_{\cC})\sqcup i_2(S_{\cC}))^{-1}]\to\Cyl(\cD)[(i_1(S_{\cD})\sqcup i_2(S_{\cD}))^{-1}]\]
	between cylinder objects of $\cC[S_{\cC}^{-1}]$ and $\cD[S_{\cD}^{-1}]$ given by Theorem \ref{thm:cylinder-loc}, which is an extension of the dg functor
	\[F\amalg F\colon \cC[S_{\cC}^{-1}]\amalg\cC[S_{\cC}^{-1}] \to \cD[S_{\cD}^{-1}]\amalg\cD[S_{\cD}^{-1}]\]
	by additionally specifying
	\begin{align*}
		t_C, t_C',\hat t_C,\check t_C,\bar t_C&\mapsto t_{F(C)}, t_{F(C)}',\hat t_{F(C)},\check t_{F(C)},\bar t_{F(C)}&&\text{respectively, for each object $C\in\cC$}\\
		t_f&\mapsto t_{F(f)} &&\text{for each generating morphism $f$ in $\cC$}
	\end{align*}
	where $t_{F(f)}$ is defined according to Definition \ref{dfn:generalised-t-morphisms}, and for any $g\colon A\to B$ in $S_{\cD}$, we define the morphisms
	\begin{align}\label{eq:t-theta-loc}
		t_{g'} &:= -i_2(g')\circ t_g\circ i_1(g')-i_2(\hat g)\circ t_A\circ i_1(g')+ i_2(g')\circ t_B\circ  i_1(\check g)+i_2(\hat g\circ g'-g'\circ\check g)\circ t_B\\
		t_{\hat g} &:=i_2(g')\circ t_g\circ i_1(\hat g)+i_2(\hat g)\circ t_A\circ i_1(\hat g)+i_2(g')\circ t_B\circ i_1(\bar g)-i_2(\hat g\circ \hat g+g'\circ\bar g)\circ t_A\notag\\
		&\hspace{11cm} -i_2(\hat g\circ g' - g'\circ \check g)\circ t_g\notag\\
		t_{\check g} &:= - i_2(\check g)\circ t_g\circ i_1(g')+ i_2(\check g)\circ t_B\circ i_1(\check g)+ i_2(\bar g)\circ t_A\circ i_1(g')-i_2(\check g\circ \check g+\bar g\circ g')\circ t_B\notag\\
		t_{\bar g} &:= - i_2(\check g)\circ t_g\circ i_1(\hat g)- i_2(\check g)\circ t_B\circ i_1(\bar g)+ i_2(\bar g)\circ t_A\circ i_1(\hat g)+i_2(\check g\circ \bar g-\bar g\circ\hat g)\circ t_A\notag\\
		&\hspace{11cm} -i_2(\check g\circ \check g + \bar g\circ g')\circ t_g\notag
	\end{align}
	in $\Cyl(\cD)[(i_1(S_{\cD})\sqcup i_2(S_{\cD}))^{-1}]$.
	
	\item The dg functor $\Cyl(F)$ satisfies
	\begin{equation*}\Cyl(F)(t_{\theta})=t_{F(\theta)}\end{equation*}
	for any morphism $\theta$ in $\cC[S_{\cC}^{-1}]$.
	
	\item The diagram
	\begin{equation*}\begin{tikzcd}
			\cC[S_{\cC}^{-1}]\amalg\cC[S_{\cC}^{-1}]\ar[r,"F\amalg F"]\ar[d,"i_{\cC[S_{\cC}^{-1}]}"] & \cD[S_{\cD}^{-1}]\amalg\cD[S_{\cD}^{-1}]\ar[d,"i_{\cD[S_{\cD}^{-1}]}"]\\
			\Cyl(\cC)[(i_1(S_{\cC})\sqcup i_2(S_{\cC}))^{-1}] \ar[r,"\Cyl(F)"]\ar[d,"p_{\cC[S_{\cC}^{-1}]}"] & \Cyl(\cD)[(i_1(S_{\cD})\sqcup i_2(S_{\cD}))^{-1}]\ar[d,"p_{\cD[S_{\cD}^{-1}]}"]\\
			\cC[S_{\cC}^{-1}]\ar[r,"F"] & \cD[S_{\cD}^{-1}]
	\end{tikzcd}\end{equation*}
	commutes.
	\end{enumerate}
\end{thm}

\begin{proof}
	First, it is straightforward (although tedious) to check that \eqref{eq:theta-differential} holds for $t_{g'},t_{\hat g},t_{\check g},t_{\bar g}$ for every morphism $g$ in $S_{\cD}$ (and $S_{\cC}$). Then, one can similarly prove Proposition \ref{prp:properties-t-theta} for the morphisms in $\cD[S_{\cD}^{-1}]$ (and $\cC[S_{\cC}^{-1}]$). Hence, the first two items of Theorem \ref{thm:cylinder-morphism-loc} can be proven similar to Theorem \ref{thm:cylinder-morphism}. The last item, the commutation of the diagram, is straightforward to check.
\end{proof}

\section{Homotopy colimit functor on the diagrams of semifree dg categories}\label{sec:hocolim-semifree}

	Our goal in this section is to construct the homotopy colimit functor $\Ho(\dgCat^J)\to\Ho(\dgCat)$, where $\dgCat$ is the model category of dg categories with Dwyer-Kan, quasi-equiconic, or Morita model structure (see Theorem \ref{thm:dgcat-model}), and $J$ is a category given as follows:
	\[a\leftarrow c\rightarrow b\]
	Note that we equip $\dgCat^J$ with the Reedy model structure as in Proposition \ref{prp:reedy}.
	
	To achieve our goal, we will explicitly construct the homotopy colimit functor on $(\dgCat_s)^J$ in Theorem \ref{thm:hocolim-functor-dg} and Theorem \ref{thm:hocolim-functor-dg-loc}, where $\dgCat_s$ is the full subcategory of $\dgCat$ consisting of semifree dg categories. The category $(\dgCat_s)^J$ can be seen as a subcategory of $\Ho(\dgCat^J)$, and there is a way to lift the homotopy colimit functor on $(\dgCat_s)^J$ to $\Ho(\dgCat^J)$ via the commuting diagram
	\[\begin{tikzcd}
		(\dgCat_s)^J\ar[r,"j",hookrightarrow]\ar[rd,"l\circ j"] &\dgCat^J\ar[r,"T"]\ar[d,"l"] & \dgCat^J\ar[r,"Q"] & \pi(\dgCat^J)_c\ar[d,"l\,\circ\,\colim"] \\
		& \Ho(\dgCat^J)\ar[rr,"\hocolim"] & & \Ho(\dgCat)
	\end{tikzcd}\]
	where $j$ is the inclusion functor, $T$ is the functor given in Proposition \ref{prp:diagram-diagonal}, and $Q$ is a cofibrant resolution functor, which we will give in Lemma \ref{lem:cofibrant-resolution-dg}. The reason is as follows: $\hocolim$ (up to natural equivalence) is the unique lift of the composition $l\circ\colim\circ Q\circ T$ by Corollary \ref{cor:dwyer-hocolim}, and each object of $\dgCat^J$ is weakly equivalent to an object of $(\dgCat_s)^J$.

	Therefore, we only need to describe a cofibrant resolution functor $Q$ on the image of the functor $T\circ j$:
	
	\begin{lem}\label{lem:cofibrant-resolution-dg}
		There is a cofibrant resolution functor $Q\colon(\dgCat)^J\to\pi(\dgCat^J)_c$ satisfying the following:
		\begin{enumerate}
			\item $Q$ sends an object of the form
			\[X:=(\cA\amalg \cB\xleftarrow{\alpha\amalg\beta}\cC\amalg \cC\xrightarrow{\nabla_{\cC}}\cC)\]
			where $\cA,\cB,\cC$ are semifree dg categories and $\alpha,\beta$ are dg functors, to the object
			\[Q(X):=(\cA\amalg \cB\xleftarrow{\alpha\amalg\beta}\cC\amalg \cC\xrightarrow{i_{\cC}}\Cyl(\cC))\]
			where $\Cyl(\cC)$ is defined as in Definition \ref{dfn:cylinder-pre-dg} (along with the functors $i_{\cC}$ and $p_{\cC}$ given in Theorem \ref{thm:cylinder-object}).
			
			\item $Q$ sends a morphism of the form
			\[\begin{tikzcd}
				X\ar[d,"F",blue]\\
				X'
			\end{tikzcd}
			\quad:=\quad
			\begin{tikzcd}
				\cA\amalg \cB\ar[d,"F_{\cA}\amalg F_{\cB}",blue] & \cC\amalg \cC\ar[l,"\alpha\amalg\beta"']\ar[r,"\nabla_{\cC}"]\ar[d,"F_{\cC}\amalg F_{\cC}",blue] & \cC\ar[d,"F_{\cC}",blue]\\
				\cA'\amalg \cB' & \cC'\amalg \cC'\ar[l,"\alpha'\amalg\beta'"']\ar[r,"\nabla_{\cC'}"] & \cC'
			\end{tikzcd}\]
			where $F_{\cA},F_{\cB},F_{\cC}$ are dg functors, to the morphism
			\[\begin{tikzcd}
				Q(X)\ar[d,"Q(F)",blue]\\
				Q(X')
			\end{tikzcd}
			\quad:=\quad
			\begin{tikzcd}
				\cA\amalg \cB\ar[d,"F_{\cA}\amalg F_{\cB}",blue] & \cC\amalg \cC\ar[l,"\alpha\amalg\beta"']\ar[r,"i_{\cC}"]\ar[d,"F_{\cC}\amalg F_{\cC}",blue] & \Cyl(\cC)\ar[d,"\Cyl(F_{\cC})",blue]\\
				\cA'\amalg \cB' & \cC'\amalg \cC'\ar[l,"\alpha'\amalg\beta'"']\ar[r,"i_{\cC'}"] & \Cyl(\cC')
			\end{tikzcd}\]
			where $\Cyl(F_{\cC})$ is defined as in Theorem \ref{thm:cylinder-morphism}.
		\end{enumerate}
	\end{lem}

	\begin{proof}
		For every object $X$ given in the lemma, $Q(X)$ is a cofibrant object by Proposition \ref{prp:reedy} and Theorem \ref{thm:cylinder-object}. Moreover, consider the morphism $p_X\colon Q(X)\to X$ given by
		\[\begin{tikzcd}
			Q(X)\ar[d,"p_X",blue]\\
			X
		\end{tikzcd}
		\quad:=\quad
		\begin{tikzcd}
			\cA\amalg \cB\ar[d,"1_{\cA}\amalg 1_{\cB}",blue] & \cC\amalg \cC\ar[l,"\alpha\amalg\beta"']\ar[r,"i_{\cC}"]\ar[d,"1_{\cC}\amalg 1_{\cC}",blue] & \Cyl(\cC)\ar[d,"p_{\cC}",blue]\\
			\cA\amalg \cB & \cC\amalg \cC\ar[l,"\alpha\amalg\beta"']\ar[r,"\nabla_{\cC}"] & \cC
		\end{tikzcd}\]
		which is indeed a morphism and an acyclic (Reedy) fibration (see \cite{hirschhorn}) since $p_{\cC}\circ i_{\cC}=\nabla_{\cC}$ and $p_{\cC}$ is an acyclic fibration by Theorem \ref{thm:cylinder-object}. Then, the diagram
		\[\begin{tikzcd}
			Q(X)\ar[r,"Q(F)"]\ar[d,"p_X"] & Q(X')\ar[d,"p_{X'}"]\\
			X\ar[r,"F"] & X'
		\end{tikzcd}\]
		commutes for every morphism $F\colon X\to X'$ given in the lemma since the diagram (\ref{eq:cylinder-morphism-diagram}) commutes by Theorem \ref{thm:cylinder-morphism}. Note that $Q(F)$ is indeed a morphisms also because the diagram (\ref{eq:cylinder-morphism-diagram}) commutes. Therefore, by Proposition \ref{prp:cofibrant-resolution} and Definition \ref{dfn:cof-res-functor}, there is a cofibration functor $Q$ with the properties given in the lemma.
	\end{proof}

	\begin{rmk}\label{rmk:cof-res-localisation}
		Assume that in Lemma \ref{lem:cofibrant-resolution-dg}, we replace $\cC$ and $\cC'$ by $\cC[S_{\cC}^{-1}]$ and $\cC'[S_{\cC'}^{-1}]$ for some subsets of degree zero closed morphisms $S_{\cC}$ and $S_{\cC'}$ in $\cC$ and $\cC'$, respectively. Then, instead of using $\Cyl(F_{\cC[S_{\cC}^{-1}]})\colon \Cyl\left(\cC[S_{\cC}^{-1}]\right)\to\Cyl(\cC'[S_{\cC'}^{-1}])$ when defining $Q$, we can use the dg functor
		\[\Cyl(\cC)[(i_1(S_{\cC})\sqcup i_2(S_{\cC}))^{-1}]\to\Cyl(\cC')[(i_1(S_{\cC'})\sqcup i_2(S_{\cC'}))^{-1}]\]
		given in Theorem \ref{thm:cylinder-morphism-loc}, which gives a simpler cofibrant resolution functor $Q$. The proof will be the same as the proof of Lemma \ref{lem:cofibrant-resolution-dg} after replacing Theorem \ref{thm:cylinder-morphism} with Theorem \ref{thm:cylinder-morphism-loc} in the proof.
	\end{rmk}

	Finally, we can state the formula for the homotopy colimit functor on the diagrams of semifree dg categories:
	
	\begin{thm}\label{thm:hocolim-functor-dg}
		The homotopy colimit functor $\hocolim\colon\Ho(\dgCat^J)\to\Ho(\dgCat)$ (up to natural equivalence) satisfies the following:
		\begin{enumerate}
			\item \label{item:hocolim-obj} $\hocolim$ sends an object of the form
			\[X:=(\cA\xleftarrow{\alpha}\cC\xrightarrow{\beta}\cB)\]
			where $\cA,\cB,\cC$ are semifree dg categories and $\alpha,\beta$ are dg functors, to the object (semifree dg category)
			\[\hocolim(X)\]
			which is the semifree extension of $\cA\amalg\cB$ by
			\begin{itemize}
				\item closed degree zero morphism $t_C\colon \alpha(C)\to\beta(C)$ for each object $C\in\cC$,
				
				\item morphisms $t_C',\hat t_C,\check t_C,\bar t_C$ for each object $C\in\cC$ as in Definition \ref{dfn:cylinder-pre-dg} (after replacing $i_1$ with $\alpha$ and $i_2$ with $\beta$),
				
				\item degree $|f|-1$ morphism $t_f\colon\alpha(A)\to\beta(B)$ for each generating morphism $f\colon A\to B$ in $\cC$ whose differential is given as in Definition \ref{dfn:cylinder-pre-dg} (after replacing $i_1$ with $\alpha$ and $i_2$ with $\beta$).
			\end{itemize}
			
			\item \label{item:hocolim-mor} $\hocolim$ sends a morphism of the form
			\[\begin{tikzcd}
				X\ar[d,"F",blue]\\
				X'
			\end{tikzcd}
			\quad:=\quad
			\begin{tikzcd}
				\cA\ar[d,"F_{\cA}",blue] & \cC\ar[l,"\alpha"']\ar[r,"\beta"]\ar[d,"F_{\cC}",blue] & \cB\ar[d,"F_{\cB}",blue]\\
				\cA' & \cC'\ar[l,"\alpha'"']\ar[r,"\beta'"] & \cB'
			\end{tikzcd}\]
			where $F_{\cA},F_{\cB},F_{\cC}$ are dg functors, to the morphism (dg functor) 
			\begin{align*}
				\hocolim(F)\colon\hocolim(X)&\to\hocolim(X')\\
				A&\mapsto F_{\cA}(A) &&\text{for any object $A\in\cA$}\\
				B&\mapsto F_{\cB}(B) &&\text{for any object $B\in\cB$}\\
				a&\mapsto F_{\cA}(a) &&\text{for any morphism $a$ in $\cA$}\\
				b&\mapsto F_{\cB}(b) &&\text{for any morphism $b$ in $\cB$}\\
				t_C,t_C',\hat t_C,\check t_C,\bar t_C&\mapsto t_{F_{\cC}(C)},t_{F_{\cC}(C)}',\hat t_{F_{\cC}(C)},\check t_{F_{\cC}(C)},\bar t_{F_{\cC}(C)} &&\text{respectively, for any object $C\in\cC$}\\
				t_f &\mapsto t_{F_{\cC}(f)} &&\text{for any generating morphism $f$ in $\cC$,}
			\end{align*}
			where for any generating morphism $f\colon A\to B$ in $\cC$, the degree $|f|-1$ morphism $t_{F_{\cC}(f)}$ is defined as
			\[t_{F_{\cC}(f)}:=\sum_{i=1}^m c_i \sum_{j=1}^{n_i}(-1)^{|f_{i,j-1}|+\ldots+|f_{i,1}|} \beta'(f_{i,n_i})\circ\ldots\circ \beta'(f_{i,j+1})\circ t_{f_{i,j}}\circ \alpha'(f_{i,j-1})\circ\ldots\circ \alpha'(f_{i,1})\]
			if $F_{\cC}(f)$ is given by
			\[F_{\cC}(f)=c 1_{F_{\cC}(A)} + \sum_{i=1}^m c_i f_{i,n_i}\circ\ldots\circ f_{i,j}\circ\ldots\circ f_{i,1}\in\hom_{\cC'}^*(F_{\cC}(A),F_{\cC}(B))\]
			for some generating morphisms $f_{i,j}$ of $\cC'$ and $c, c_i\in k$ ($c=0$ if $F_{\cC}(A)\neq F_{\cC}(B)$).
		\end{enumerate}
	\end{thm}

	\begin{proof}
		First, we note that the description of $\hocolim(X)$ is given in \cite{hocolim}. We can also see it here as follows: By Corollary \ref{cor:dwyer-hocolim}, we have
		\[\hocolim(X)=l\circ\colim\circ Q\circ T(X)\]
		where $T$ is a functor as in Proposition \ref{prp:diagram-diagonal}, $Q$ is the cofibration functor given in Lemma \ref{lem:cofibrant-resolution-dg}. Using the description of $\colim$ given in Proposition \ref{prp:colimit-dg}, it is straightforward to check that $\hocolim(X)$ is as described in the first item.
		
		For the second item (which does not appear in \cite{hocolim}), by Corollary \ref{cor:dwyer-hocolim} again, we have
		\[\hocolim(F)=l\circ\colim\circ Q\circ T(F) .\]
		Then, considering that $\hocolim(X)$ is the semifree extension of $\cA\amalg\cB$ by the morphims given in the theorem, we see that $\hocolim(F)\colon\hocolim(X)\to\hocolim(X')$ acts like $F_{\cA}$ on $\cA$, $F_{\cB}$ on $\cB$, and $\Cyl(F_{\cC})$ (which is given in Theorem \ref{thm:cylinder-morphism}) on the added morphisms.
	\end{proof}
	
	\begin{rmk}
		In Theorem \ref{thm:hocolim-functor-dg}, we could have expressed $\hocolim(X)$ as the semifree dg category obtained by first taking the semifree extension of $\cA\amalg\cB$ by the morphisms $t_C$ and $t_f$ for each $C\in\cC$ and for each generating morphism $f$ in $\cC$, and then taking the dg localization of the resulting category at the morphisms $\{t_C\vb C\in\cC\}$ as in Remark \ref{rmk:cylinder-dg-loc}. Hence, up to natural equivalence, the images of the morphisms $t'_C,\hat t_C,\check t_C,\bar t_C$ under the homotopy colimit functor are uniquely determined for every $C\in\cC$.
	\end{rmk}

	\begin{rmk}
		Theorem \ref{thm:hocolim-functor-dg} implies that for any morphism $\theta$ in $\cC$, we have
		\[\hocolim(F)(t_{\theta})=t_{F_{\cC}(\theta)} .\]
		This follows from the identity (\ref{eq:cyl-theta}).
	\end{rmk}
	
	\begin{rmk}\label{rmk:hocolim-dg-algebras}
		Similar to Remark \ref{rmk:cylinder-functor-dg-algebras}, we can modify Theorem \ref{thm:hocolim-functor-dg} to describe the homotopy colimit functor $\Ho(\dgAlg^J)\to\Ho(\dgAlg)$ for the model category of dg algebras $\dgAlg$. All the formulas still hold by setting
		\[t_C=1,\quad t_C'=1,\quad \hat t_C=0,\quad\check t_C=0,\quad \bar t_C=0.\]
	\end{rmk}

	Considering Theorem \ref{thm:cylinder-morphism-loc} and Remark \ref{rmk:cof-res-localisation}, we have the following version of Theorem \ref{thm:hocolim-functor-dg} when $\cC$ is given as a dg localization:
	
	\begin{thm}\label{thm:hocolim-functor-dg-loc}
		The homotopy colimit functor $\hocolim\colon\Ho(\dgCat^J)\to\Ho(\dgCat)$ (up to natural equivalence) satisfies the following:
		\begin{enumerate}
			\item \label{item:hocolim-loc-obj} $\hocolim$ sends an object of the form
			\[X:=(\cA\xleftarrow{\alpha}\cC[S_{\cC}^{-1}]\xrightarrow{\beta}\cB)\]
			where $\cA,\cB,\cC$ are semifree dg categories, $S_{\cC}$ is a subset of degree zero morphisms in $\cC$, and $\alpha,\beta$ are dg functors, to the object (semifree dg category)
			\[\hocolim(X)=\hocolim(\cA\xleftarrow{\alpha}\cC\xrightarrow{\beta}\cB)\]
			which is given by Theorem \ref{thm:hocolim-functor-dg}.
			
			\item \label{item:hocolim-loc-mor} $\hocolim$ sends a morphism of the form
			\[\begin{tikzcd}
				X\ar[d,"F",blue]\\
				X'
			\end{tikzcd}
			\quad:=\quad
			\begin{tikzcd}
				\cA\ar[d,"F_{\cA}",blue] & \cC[S_{\cC}^{-1}]\ar[l,"\alpha"']\ar[r,"\beta"]\ar[d,"F_{\cC}",blue] & \cB\ar[d,"F_{\cB}",blue]\\
				\cA' & \cC'[S_{\cC'}^{-1}]\ar[l,"\alpha'"']\ar[r,"\beta'"] & \cB'
			\end{tikzcd}\]
			where for any generating morphism $f\colon A\to B$ in $\cC$, the degree $|f|-1$ morphism $t_{F_{\cC}(f)}$ is defined as
			\[t_{F_{\cC}(f)}:=\sum_{i=1}^m c_i \sum_{j=1}^{n_i}(-1)^{|f_{i,j-1}|+\ldots+|f_{i,1}|} \beta'(f_{i,n_i})\circ\ldots\circ \beta'(f_{i,j+1})\circ t_{f_{i,j}}\circ \alpha'(f_{i,j-1})\circ\ldots\circ \alpha'(f_{i,1})\]
			if $F_{\cC}(f)$ is given by
			\[F_{\cC}(f)=c 1_{F_{\cC}(A)} + \sum_{i=1}^m c_i f_{i,n_i}\circ\ldots\circ f_{i,j}\circ\ldots\circ f_{i,1}\in\hom_{\cC'[S_{\cC'}^{-1}]}^*(F_{\cC}(A),F_{\cC}(B))\]
			for some generating morphisms $f_{i,j}$ of $\cC'[S_{\cC'}^{-1}]$ and $c, c_i\in k$ ($c=0$ if $F_{\cC}(A)\neq F_{\cC}(B)$). The generating morphisms of $\cC'[S_{\cC'}^{-1}]$ are given by the generating morphisms of $\cC'$ and the morphisms $g',\hat g,\check g, \bar g$ for every morphism $g$ in $S_{\cC'}$ as in Proposition \ref{prp:localise-semifree}. The morphisms
			\[t_{g'},t_{\hat g},t_{\check g},t_{\bar g}\]
			are given in terms of $t_g$ as in \eqref{eq:t-theta-loc} and $t_g$ is given as in Definition \ref{dfn:generalised-t-morphisms} (after replacing $i_1$ with $\alpha'$ and $i_2$ with $\beta'$).
		\end{enumerate}
	\end{thm}

	\begin{proof}
		The first item appears in \cite{hocolim}. We can also see it here as in the proof of Theorem \ref{thm:hocolim-functor-dg} after replacing the cofibrant resolution functor $Q$ with the one given in Remark \ref{rmk:cof-res-localisation}.
		
		The second item (which does not appear in \cite{hocolim}) can be again shown as in the proof of Theorem \ref{thm:hocolim-functor-dg} after replacing the cofibrant resolution functor $Q$ with the one given in Remark \ref{rmk:cof-res-localisation}. In particular, we replace the functor $\Cyl(F)$ given in Theorem \ref{thm:cylinder-morphism} to define $Q$ with the one given in Theorem \ref{thm:cylinder-morphism-loc}.
	\end{proof}
	
	\begin{rmk}\label{rmk:hocolim-non-semifree}
		According to \cite{holstein-properness}, if the coefficient ring $k$ has flat dimension 0 (e.g., $k$ is a field), then the model category $\dgCat$ with weak equivalences being quasi-equivalences, pretriangulated equivalences, or Morita equivalences is left proper. This implies that the formulas in Theorems \ref{thm:hocolim-functor-dg} and \ref{thm:hocolim-functor-dg-loc} still hold if we choose arbitrary dg categories $\cA,\cA',\cB,\cB'$ instead of semifree ones, given that $k$ has flat dimension 0.
	\end{rmk}

	In Section \ref{sec:sphere-reflection}, we will demonstrate a simple application of Theorem \ref{thm:hocolim-functor-dg} and Theorem \ref{thm:hocolim-functor-dg-loc} to symplectic geometry.

\section{I-category and cofibration category of semifree dg categories}\label{sec:cofibration-category}

The purpose of this section is to elaborate on the construction of the cylinder functor introduced in Section \ref{sec:cylinder-functor} to approach the homotopy theory of dg categories up to quasi-equivalence. This approach is independent of the celebrated result by Tabuada (Theorem \ref{thm:dgcat-model}\eqref{item:dwyer-kan-model}). The advantage of our approach is that the simplicity of our cylinder functor makes many constructions explicit and computable. It is worth noting that readers can skip to Section \ref{sec:sphere-reflection} without losing continuity in the paper.

We will define an $I$-category structure on the category of semifree dg categories in Theorem \ref{thm:dgcat-I-category} by introducing cofibrations and a natural cylinder functor. This structure enables us to establish a cofibration category of semifree dg categories in Theorem \ref{thm:dgcat-cofibration-category}, following the framework outlined in \cite{baues}. Conceptually, this category corresponds to half of a model category. Alternatively, it can be regarded as a category of cofibrant objects first introduced by Brown in \cite{brown-homotopy} and further studied, for example, in \cite{kamps-porter}.

These frameworks enable the construction of various homotopical operations, such as homotopy colimits (Remark \ref{rmk:hocolim-from-cofibration}), without relying on the existence of a model structure. Thus, it provides an alternative pathway to achieve the most of the results outlined in the preceding sections without relying on \cite{tabuada-model}. We will also present an explicit functorial factorization of dg functors into a cofibration and a quasi-equivalence in Theorem \ref{thm:mapping-cylinder-factorisation}.

We recall that $k$ is a commutative ring, and $\dgCat_s$ represents the category of (small) $k$-linear semifree dg categories, with morphisms defined as dg functors. The definitions of semifree dg categories and semifree extensions can be revisited from Definition \ref{dfn:semifree}. Note that every dg category is quasi-equivalent to a semifree one, thus justifying our focus on $\sCat$.

We now introduce the definition of an $I$-category as outlined in \cite{baues}. Then, we will show how the category of semifree dg categories can be made an $I$-category:

\begin{dfn}\label{dfn:i-category}
	An {\em $I$-category} is a category $\cM$ with an initial object $\emptyset$ and equipped with a structure $(\mathrm{cof}, I)$. The term $\mathrm{cof}$ represents a class of morphisms in $\cM$ identified as cofibrations, and $I$ is a functor from $\cM$ to itself. The structure satisfies the following set of axioms:
	\begin{enumerate}
		\item {\em Cylinder Axiom:} $I\colon \cM \to \cM$ is a functor together with natural transformations
		\[i_1, i_2\colon \mathrm{id}_{\cM} \Rightarrow I, \quad p\colon I \Rightarrow\mathrm{id}_{\cM},\]
		such that for all $\cA\in\cM$ and $r\in\{1,2\}$, $p\circ i_r\colon \cA\to I(\cA)\to \cA$ is the identity of $\cA$.
		
		\item\label{item:pushout-axiom} {\em Pushout Axiom:} For a cofibration $F\colon \cA \to \cB$ and a morphism $G\colon \cA\to \cC$ in $\cM$, there exists a pushout in $\cM$
		\[
		\begin{tikzcd}
			\cB\ar[r,"\bar G"] & \cB \cup_{\cA} \cC \\
			\cA\ar[u,"F"]\arrow[r, "G"] & \cC \arrow[u, "\bar F"]
		\end{tikzcd}
		\]
		where $\bar F$ is also a cofibration. Moreover, $I(\emptyset) = \emptyset$, and $I$ sends the pushout diagram into a pushout diagram, i.e., $I(\cB \cup_{\cA} \cC) = I(\cB) \cup_{I(\cA)} I(\cC)$.
		
		\item\label{item:cofibration-axiom} {\em Cofibration Axiom:} Each isomorphism in $\cM$ is a cofibration, and for each $\cA\in\cM$, the morphism $\emptyset \to \cA$ is a cofibration. The composition of cofibrations is a cofibration. Moreover, any cofibration $F\colon \cA \to \cB$ in $\cM$ has the following {\em homotopy extension property}: For each $r \in \{1,2\}$, and for each commutative diagram in $\cM$
		\[\begin{tikzcd}
			\cB \arrow[r, "G"] & \cC \\
			\cA \arrow[r, "i_r"] \arrow[u, "F"] & I(\cA)  \arrow[u, "H"]
		\end{tikzcd}\]
		there is a morphism $E\colon I(\cB) \to \cC$ with $E\circ i_r = G$ and $E\circ I(F) = H$.
		
		\item\label{item:relative-cylinder-axiom} {\em Relative Cylinder Axiom:} For a cofibration $F\colon \cA \to \cB$ in $\cM$, the morphism $G$ defined by the pushout diagram
		\[\begin{tikzcd}
			& \cB\sqcup \cB\ar[rd]\ar[rrrd,"i_1\sqcup i_2",bend left=10]\\
			\cA\sqcup \cA\ar[ru,"F\sqcup F"]\ar[rd,"i_1\sqcup i_2"'] & \mathrm{pushout} & \cB\cup_{\cA} I(\cA)\cup_{\cA} \cB\ar[rr,"G",dashed] & &I(\cB)\\
			& I(\cA)\ar[ru]\ar[rrru,"I(F)"',bend right=10]
		\end{tikzcd}\]
		is a cofibration.
		
		\item {\em Interchange Axiom:} For all $\cA\in\cM$, there exists a morphism $T\colon I(I(\cA)) \to I(I(\cA))$ such that $T\circ i_r = I(i_r)$ and $T\circ I(i_r)=i_r$ for any $r\in\{1,2\}$.
	\end{enumerate}
\end{dfn}


\begin{thm}\label{thm:dgcat-I-category}
	The category of semifree dg categories $\sCat$ is an $I$-category with the structure $(\mathrm{cof},I)$, which is defined as follows:
	\begin{itemize}
		\item $\mathrm{cof}:$ Cofibrations are semifree extensions.
		
		\item $I$: The functor $I$ is the cylinder functor $\Cyl\colon\sCat\to\sCat$ introduced in Corollary \ref{cor:cyl-functor}.
	\end{itemize}
\end{thm}

\begin{proof}
	First, note that the empty category $\emptyset$ is an initial object of $\sCat$. We introduce the following conventions and review some definitions:
	\begin{itemize}
		\item Given a semifree dg category $\cA$, we fix a set of generating morphisms $\{a_i\colon X_i\to Y_i\}$ (indexed by an ordinal). Then, the semifree dg category $\Cyl(\cA)$ is given as follows:
		\begin{enumerate}[label = (\roman*)]
			\item {\em Objects:} $A^1,A^2$ for each $A\in\cA$.
			\item {\em Generating morphisms:} For each object $A\in\cA$,
			\[t_A\colon A^1\to A^2,\quad t_A'\colon A^2\to A^1,\quad \hat t_A\colon A^1\to A^1,\quad \check t_A\colon A^2\to A^2,\quad \bar t_A\colon A^1\to A^2,\]
			and for each generating morphism $a_i\colon X_i\to Y_i$ of $\cA$,
			\[a_i^1\colon X_i^1\to Y_i^1,\quad a_i^2\colon X_i^2\to Y_i^2,\quad t_{a_i}\colon X_i^1\to Y_i^2 .\]
			\item {\em Degrees:} 
			\[|t_A|=|t'_A|=0, \quad |\hat t_A|=|\check t_A|=-1,\quad |\bar t_A|=-2,\quad |a_i^1|=|a_i^2|=|a_i|,\quad |t_{a_i}|=|a_i|-1 .\]
			\item {\em Differentials:}
			\begin{gather*}
				dt_A=dt'_A=0,\quad d\hat t_A=1_{A^1}-t'_A\circ t_A,\quad d\check t_A=1_{A^2}-t_A\circ t'_A,\quad d\bar t_A=t_A\circ\hat t_A-\check t_A\circ t_A ,\\
				da_i^1=(da_i)^1,\quad da_i^2=(da_i)^2,\quad
				dt_{a_i}=(-1)^{|a_i|}(a_i^2\circ t_{X_i} - t_{Y_i}\circ a_i^1)+ t_{da_i} ,
			\end{gather*}
			where $t_{da_i}$ is defined according to Definition \ref{dfn:generalised-t-morphisms}, and for each $r\in\{1,2\}$, $(da_i)^r$ is defined via the dg functor
			\begin{gather*}
				i_r\colon \cA\to\Cyl(\cA)\\
				A\mapsto A^r,\quad a_i\mapsto a_i^r
			\end{gather*}
			by setting $a^r:=i_r(a)$ for each morphism $a$ of $\cA$.
		\end{enumerate}
		
		\item Given a semifree extension $F\colon \cA\to\cB$ by a set of objects $R_b$ and a set of morphisms $S_b=\{b_i\colon U_i\to V_i\}$, we select a set of generating morphisms of $\cB$ as
		\[\{F(a_i)\vb a_i\text{ is a generating morphism of }\cA\}\cup S_b .\]
		It is important to note that the functor $\Cyl$ may initially depend on the choice of generating morphisms, but it ultimately does not (up to natural isomorphism), as discussed in Remark \ref{rmk:cyl-functor-well-defined}.
		
		\item Given a dg functor $G\colon\cA\to\cC$ between semifree dg categories, the dg functor $\Cyl(G)$ is defined as follows:
		\begin{gather*}
			\Cyl(G)\colon\Cyl(\cA)\to\Cyl(\cC)\\
			A^r\mapsto G(A)^r,\quad a_i^r\mapsto G(a_i)^r,\quad t_{a_i}\mapsto t_{G(a_i)},\\
			t_A\mapsto t_{G(A)},\quad t'_A\mapsto t'_{G(A)},\quad \hat t_A\mapsto \hat t_{G(A)},\quad \check t_A\mapsto \check t_{G(A)}, \quad \bar t_A\mapsto \bar t_{G(A)},
		\end{gather*}
		where $r\in\{1,2\}$.
	\end{itemize}
		
	With the notations established, we proceed to verify the axioms outlined in Definition \ref{dfn:i-category} for $\cM=\sCat$ with cofibrations defined as semifree extensions and $I=\Cyl$:
	\begin{enumerate}
		\item {\em Cylinder Axiom:} The functor $I=\Cyl$ satisfies the conditions of Definition \ref{dfn:cylinder-functor} by Corollary \ref{cor:cyl-functor}, thus fulfilling the Cylinder Axiom. Note that Definition \ref{dfn:cylinder-functor} provides more than what is necessary for the axiom.
		
		\item {\em Pushout Axiom:} The existence of the pushout in Definition \ref{dfn:i-category}\eqref{item:pushout-axiom} is given by Proposition \ref{prp:colimit-dg}. Also, we have $\Cyl(\emptyset)=\emptyset$ by definition. Moreover, if $F\colon\cA\to\cB$ is a semifree extension by a set of objects $R_b$ and a set of morphisms $S_b$, then $\Cyl(F)\colon \Cyl(\cA)\to\Cyl(\cB)$ is a semifree extension by the set of objects $\{B^1,B^2\vb B\in R_b\}$ and the set of morphisms
		\[\{t_{B},t'_{B},\hat t_{B},\check t_{B},\bar t_{B}\vb B\in R_b\}\cup \{b_i^1,b_i^2,t_{b_i}\vb b_i\in S_b\} .\]
		Then, given a pushout square
		\[\begin{tikzcd}
			\cB\ar[r,"\bar G"] & \cD \\
			\cA\ar[u,"F"]\arrow[r, "G"] & \cC \arrow[u, "\bar F"]
		\end{tikzcd}\]
		as in Proposition \ref{prp:colimit-dg}, applying the functor $\Cyl$ yields a commutative diagram
		\[\begin{tikzcd}
			\Cyl(\cB)\ar[r,"\Cyl(\bar G)"] & \Cyl(\cD) \\
			\Cyl(\cA)\ar[u,"\Cyl(F)"]\arrow[r, "\Cyl(G)"] & \Cyl(\cC) \arrow[u, "\Cyl(\bar F)"]
		\end{tikzcd}\]
		where $\Cyl(\bar F)$ is a semifree extension, as $\bar F$ is. Also, it is straightforward to verify that $\Cyl(\bar F)=\overline{\Cyl(F)}$ and $\Cyl(\bar G)=\overline{\Cyl(G)}$. Thus, the above diagram forms a pushout square.
		
		\item {\em Cofibration Axiom:} Clearly, each isomorphism in $\sCat$ is a semifree extension, and for each $\cA\in\sCat$, $\emptyset\to\cA$ is a semifree extension. Moreover, the composition of semifree extensions is a semifree extension.
		
		Regarding the homotopy extension property, let us first set $r=1$ without loss of generality. Consider the commutative diagram described in Definition \ref{dfn:i-category}\eqref{item:cofibration-axiom}, where $I=\Cyl$. Then, we define the dg functor $E\colon \Cyl(\cB)\to\cC$ as follows: Since $\Cyl(F)\colon\Cyl(\cA)\to\Cyl(\cB)$ is a semifree extension, we can view $\Cyl(\cA)$ as a full dg subcategory of $\Cyl(\cB)$. We define
		\[E|_{\Cyl(\cA)}:=H.\]
		Moreover, we set
		\[E(B^1)=E(B^2):=G(B),\quad E(t_B)=E(t'_B):=1_{G(B)},\quad E(\hat t_B)=E(\check t_B)=E(\bar t_B):=0\]
		for each $B\in R_b$, and
		\[E(b_i^1):=G(b_i)\]
		for each $b_i\in S_b$. We will define $E(b_i^2)$ and $E(t_{b_i})$ by transfinite induction: Let us give an ordering on the set of morphisms $S_b=\{b_i\colon U_i\to V_i\}$ (indexed by an ordinal) such that each $db_i$ in $\cB$ is generated by the morphisms of $\cA$ and $\{b_j\vb j<i\}$. Assume that $E$ is defined on the morphisms generated by the morphisms of $\Cyl(\cA)$ and
		\[\{t_{B},t'_{B},\hat t_{B},\check t_{B},\bar t_{B}\vb B\in R_b\}\cup\{b_i^1\vb b_i\in S_b\}\cup\{b_j^2,t_{b_j}\vb j<i\} .\]
		Then, we define
		\begin{align*}
			E(b_i^2)&:=E(t_{V_i}\circ b_i^1\circ t_{U_i}' - (-1)^{|b_i|}t_{db_i}\circ t'_{U_i}-(-1)^{|b_i|}(db_i)^2\circ \check t_{U_i}),\\
			E(t_{b_i})&:=E(-t_{V_i}\circ b_i^1\circ\hat t_{U_i} + (-1)^{|b_i|}t_{db_i}\circ\hat t_{U_i} - (-1)^{|b_i|}(db_i)^2\circ \bar t_{U_i}) .
		\end{align*}
		It is straightforward to check that $E$ is a dg functor, as $dE(b_i^2)=E(db_i^2)=E((db_i)^2)$ and $dE(t_{b_i})=E(dt_{b_i})=E((-1)^{|b_i|}(b_i^2\circ t_{U_i} - t_{V_i}\circ b_i^1)+t_{db_i})$, and it satisfies $E\circ i_1 = G$ and $E\circ \Cyl(F) = H$.
		
		\item {\em Relative Cylinder Axiom:} Consider the diagram described in Definition \ref{dfn:i-category}\eqref{item:relative-cylinder-axiom}, where $I=\Cyl$. The dg category $\cB\cup_{\cA} \Cyl(\cA)\cup_{\cA}\cB$ is a semifree extension of $\Cyl(\cA)$ by the set of objects $\{B^1,B^2\vb B\in R_b\}$ and the set of morphisms $\{b_i^1,b_i^2\vb b_i\in S_b\}$. Then, the dg functor $G\colon \cB\cup_{\cA} \Cyl(\cA)\cup_{\cA} \cB\to \Cyl(\cB)$ is a semifree extension by the morphisms
		\[\{t_{B},t'_{B},\hat t_{B},\check t_{B},\bar t_{B}\vb B\in R_b\}\cup \{t_{b_i}\vb b_i\in S_b\} .\]
		Hence, $G$ is a cofibration.
		
		\item {\em Interchange Axiom:} First, note that for $s\in\{1,2\}$, we have the dg functor
		\begin{align*}
			i_s\colon\Cyl(\cA)&\to\Cyl(\Cyl(\cA))\\
			A^r&\mapsto A^{rs} &&\text{if $A^r\in\Ob(\Cyl(\cA))$ and $r\in\{1,2\}$,}\\
			\theta&\mapsto \theta^s &&\text{if $\theta\in\Mor(\Cyl(\cA))$,}
		\end{align*}
		and
		the commutative diagram
		\[\begin{tikzcd}
			\cA\ar[r,"i_r"]\ar[d,"i_s"] & \Cyl(\cA)\ar[d,"i_s"]\\
			\Cyl(\cA)\ar[r,"\Cyl(i_r)"] & \Cyl(\Cyl(\cA))
		\end{tikzcd}\]
		for any $r,s\in\{1,2\}$ by Cylinder Axiom. Then, the semifree dg category $\Cyl(\Cyl(\cA))$ can be given as follows:
		\begin{enumerate}[label = (\roman*)]
			\item {\em Objects:} $A^{rs}:=i_s(i_r(A))=\Cyl(i_r)(i_s(A))$ for each $A\in\cA$ and $r,s\in\{1,2\}$.
			\item {\em Generating morphisms:} For each object $A\in\cA$ and $r\in\{1,2\}$,
			\begin{gather*}
				(t^*_A)^r:=i_r(t^*_A),\quad t^*_{A^r}:=\Cyl(i_r)(t^*_A),\\
				t_{t_A}\colon A^{11}\to A^{22},\quad t_{t'_A}\colon A^{21}\to A^{12},\quad t_{\hat t_A}\colon A^{11}\to A^{12},\quad t_{\check t_A}\colon A^{21}\to A^{22},\quad t_{\bar t_A}\colon A^{11}\to A^{22} ,
			\end{gather*}
			where $t^*\in\{t,t',\hat t,\check t, \bar t\}$, and for each generating morphism $a_i\colon X_i\to Y_i$ of $\cA$ and $r,s\in\{1,2\}$,
			\[a_i^{rs}:=i_s(i_r(a_i))=\Cyl(i_r)(i_s(a_i)),\quad (t_{a_i})^r:=i_r(t_{a_i}),\quad t_{a_i^r}:=\Cyl(i_r)(t_{a_i}),\quad
			t_{t_{a_i}}\colon X_i^{11}\to Y_i^{22} .\]
			\item {\em Degrees:} $|(t^*_A)^r|=|t^*_{A^r}|=|t^*_A|,\; |t_{t^*_A}|=|t^*_A|-1,\; |(t^*_{a_i})^r|=|t^*_{a_i^r}|=|a_i|-1,\; |t_{t_{a_i}}|=|a_i|-2 .$
			\item {\em Differentials:}
			\begin{gather*}
				d(t^*_A)^r=i_r(dt^*_A),\quad d(t^*_{A^r})=\Cyl(i_r)(dt^*_A),\\
				dt_{t_A}=(t_A)^2\circ t_{A^1}-t_{A^2}\circ (t_A)^1,\quad 
				dt_{t'_A}=(t_A')^2 \circ t_{A^2}-t_{A^1}\circ (t_A')^1 ,\\
				dt_{\hat t_A}=-(\hat t_A)^2\circ t_{A^1}+t_{A^1}\circ (\hat t_A)^1 -(t'_A)^2\circ t_{t_A} - 	t_{t'_A}\circ (t_A)^1,\\
				dt_{\check t_A}=-(\check t_A)^2\circ t_{A^2}+t_{A^2}\circ (\check t_A)^1 -(t_A)^2\circ 	t_{t'_A} - t_{t_A}\circ (t'_A)^1 ,\\
				dt_{\bar t_A}=(\bar t_A)^2\circ t_{A^1}-t_{A^2}\circ (\bar t_A)^1 +(t_A)^2\circ t_{\hat 	t_A}-t_{t_A}\circ(\hat t_A)^1 - (\check t_A)^2\circ t_{t_A} -t_{\check t_A}\circ (t_A)^1,\\
				da_i^{rs}=i_s(i_r(da_i))=\Cyl(i_r)(i_s(da_i)),\quad d(t_{a_i})^r=i_r(dt_{a_i}),\quad dt_{a_i^r}=\Cyl(i_r)(dt_{a_i}),\\
				dt_{t_{a_i}}=(-1)^{|a_i|}(-(t_{a_i})^2\circ t_{X_i^1}+t_{Y_i^2}\circ (t_{a_i})^1
				+a_i^{22}\circ t_{t_{X_i}} +t_{a_i^2}\circ (t_{X_i})^1
				-(t_{Y_i})^2\circ t_{a_i^1})- t_{t_{Y_i}}\circ a_i^{11} +t_{t_{da_i}} .
			\end{gather*}
		\end{enumerate}

	\end{enumerate}
	Given this expression of $\Cyl(\Cyl(\cA))$, we define a functor
	\begin{gather*}
		T\colon \Cyl(\Cyl(\cA))\to\Cyl(\Cyl(\cA))\\
		A^{rs}\mapsto A^{sr},\quad
		(t^*_A)^r\mapsto t^*_{A^r},\quad
		t^*_{A^r}\mapsto (t^*_A)^r\\
		t_{t_A}\mapsto -t_{t_A},\quad
		t_{t'_A}\mapsto t'_{A^2}\circ t_{t_A}\circ t'_{A^1} - \hat t_{A^2}\circ (t_A)^1\circ t'_{A^1} +t'_{A^2}\circ (t_A)^2\circ \check t_{A^1},\\
		t_{\hat t_A}\mapsto -t'_{A^2}\circ t_{t_A}\circ\hat t_{A^1} + \hat t_{A^2}\circ (t_A)^1\circ \hat t_{A^1} + t'_{A^2}\circ (t_A)^2\circ\bar t_{A^1} ,\\
		t_{\check t_A}\mapsto \check t_{A^2}\circ  t_{t_A}\circ t'_{A^1} +\check t_{A^2}\circ (t_A)^2\circ \check t_{A^1}+ \bar t_{A^2}\circ (t_A)^1\circ t'_{A^1},\\
		t_{\bar t_A}\mapsto \check t_{A^2}\circ t_{t_A}\circ \hat t_{A^1} + \bar t_{A^2}\circ (t_A)^1\circ \hat t_{A^1} -\check t_{A^2}\circ (t_A)^2\circ \bar t_{A^1},\\
		a_i^{rs}\mapsto a_i^{sr},\quad
		(t_{a_i})^r\mapsto t_{a_i^r},\quad
		t_{a_i^r}\mapsto (t_{a_i})^r,\quad
		t_{t_{a_i}}\mapsto -t_{t_{a_i}},
	\end{gather*}
	where $r,s\in\{1,2\}$. The functor $T$ clearly satisfies $T\circ i_r=\Cyl(i_r)$ and $T\circ \Cyl(i_r)=i_r$ for any $r\in\{1,2\}$. Confirming that $T$ is a dg functor involves a straightforward computation, with the exception of showing $dT(t_{t_{a_i}})=T(dt_{t_{a_i}})$. Upon observing that $T(t_{t_{da_i}})=-t_{t_{da_i}}$, this also becomes trivial.
\end{proof}

\begin{rmk}
	The dg functor $T$ introduced in the proof of Interchange Axiom acts as an involution as expected, except for the morphisms $t_{t'_A}, t_{\hat t_A}, t_{\check t_A}, t_{\bar t_A}$. These morphisms lack counterparts such as $t'_{t_A}, \hat t_{t_A}, \check t_{t_A}, \bar t_{t_A}$. However, they are superfluous in $\Cyl(\Cyl(\cA))$ (compare with Theorem \ref{thm:cylinder-loc}). More explicitly, $t_{t'_A}, t_{\hat t_A}, t_{\check t_A}, t_{\bar t_A}$ are homotopic to $-T(t_{t'_A}), -T(t_{\hat t_A}), -T(t_{\check t_A}), -T(t_{\bar t_A})$, respectively. Therefore, $T$ acts like an involution, but up to homotopy.
\end{rmk}

In order to show that $\sCat$ is a cofibration category using Theorem \ref{thm:dgcat-I-category}, we first need to establish an analogue of Whitehead theorem for semifree dg categories. Before presenting the theorem, we need to introduce some definitions and lemmas:

\begin{dfn}
	Let $\cM$ be an $I$-category with a structure $(\text{cof},I)$.
	\begin{enumerate}
		\item Two morphisms $F_1,F_2\colon \cA\to\cB$ in $\cM$ are {\em homotopic}, which is denoted by $F_1\approx F_2$, if there exists a morphism $H\colon I(\cA)\to\cB$ in $\cM$ such that $H\circ i_1 =F_1$ and $H\circ i_2=F_2$.
		
		\item A morphism $F\colon\cA\to\cB$ in $\cM$ is a {\em homotopy equivalence} if there exists a morphism $G\colon\cB\to\cA$ in $\cM$ such that $G\circ F\approx 1_{\cA}$ and $F\circ G\approx 1_{\cB}$.
	\end{enumerate}
\end{dfn}

\begin{lem}\label{lem:homotopy-equivalence}
	Let $\cA$ be a dg category. If $A,C\in\cA$ are isomorphic in $H^0(\cA)$, then there exist morphisms
	\[r\colon A\to C,\quad r'\colon C\to A,\quad \hat r\colon A\to A,\quad \check r\colon C\to C,\quad \bar r\colon A\to C\]
	with the gradings $|r|=|r'|=0$, $|\hat r|=|\check r|=-1$, $|\bar r|=-2$, and the differentials
	\[dr=dr'=0,\quad d\hat r=1_A-r'\circ r,\quad \check r=1_C-r\circ r',\quad d\bar r=r\circ \hat r - \check r \circ r .\]
\end{lem}

\begin{proof}
	Assume $A,C\in\cA$ are isomorphic in $H^0(\cA)$. Then, there exist morphisms
	\[s\colon A\to C,\quad s'\colon C\to A,\quad \hat s\colon A\to A,\quad \check s\colon C\to C\]
	with the gradings $|s|=|s'|=0$, $|\hat s|=|\check s|=-1$, and the differentials
	\[ds=ds'=0,\quad d\hat s=1_A-s'\circ s,\quad \check s=1_C-s\circ s' .\]
	We get the desired morphisms by defining
	\[r:=s,\quad r':=s'\circ s \circ s',\quad \hat r:=\hat s + s'\circ \check s\circ s,\quad \check r:=\check s+s\circ \hat s\circ s',\quad \bar r:=\check s\circ s\circ \hat s-s\circ\hat s\circ\hat s - \check s\circ \check s\circ s .\]
\end{proof}

\begin{lem}\label{lem:full-and-faithful}
	Let $\cA$ and $\cB$ be dg categories, and $F\colon\cA\to\cB$ be a dg functor such that $H^*(F)$ is full and faithful. If $da=0$ and $db=F(a)$ for some $a\in\Mor(\cA)$ and $b\in\Mor(\cB)$, then there exist $c\in\Mor(\cA)$ and $e\in\Mor(\cB)$ such that $dc=a$ and $de=b-F(c)$.
\end{lem}

\begin{proof}
	Since $H^*(F)$ is faithful, $da=0$ and $db=F(a)$, there exists $c_1\in\Mor(\cA)$ such that $dc_1=a$. Then, $d(b-F(c_1))=0$, and since $H^*(F)$ is full, there exist $c_2\in\Mor(\cA)$ and $e\in\Mor(\cB)$ such that $dc_2=0$ and $de=b-F(c_1)-F(c_2)$. Setting $c:=c_1+c_2$ concludes the proof.
\end{proof}

We are now prepared to present an analogue of Whitehead's theorem for semifree dg categories:

\begin{thm}\label{thm:whitehead-for-semifree}
	Consider $\sCat$ as an $I$-category with the structure $(\text{cof},I=\Cyl)$ as in Theorem \ref{thm:dgcat-I-category}. Let $F\colon\cA\to\cB$ be a morphism in $\sCat$, i.e., a dg functor between semifree dg categories. Then, $F$ is a homotopy equivalence if and only if $F$ is a quasi-equivalence.
\end{thm}

\begin{proof}
	First, let us assume that $F\colon \cA\to\cB$ is a homotopy equivalence. Then, there exists dg functors $G\colon\cB\to\cA$, $H\colon \Cyl(\cA)\to \cA$ and $K\colon \Cyl(\cB)\to \cB$ such that $H\circ i_1 =G\circ F$, $H\circ i_2=1_{\cA}$, $K\circ i_1 =F\circ G$ and $K\circ i_2=1_{\cB}$.
	\begin{itemize}
		\item $H^0(F)\colon H^0(\cA)\to H^0(\cB)$ is essentially surjective: Let $B\in\cB$. Since $K$ is a dg functor, we have the morphisms $K(t_B)\colon F(G(B))\to B$ and $K(t'_B)\colon B\to F(G(B))$ with $dK(t_B)=dK(t'_B)=0$, $dK(\hat t_B)=1_{F(G(B))}-K(t'_B)\circ K(t_B)$ and $dK(\check t_B)=1_{B}-K(t_B)\circ K(t'_B)$. Hence, $H^0(F)(G(B))=F(G(B))$ is isomorphic to $B$ in $H^0(\cB)$.
		
		\item $H^*(F)\colon H^*(\cA)\to H^*(\cB)$ is faithful: Let $a\colon A\to C$ be a morphism in $\cA$ with $da=0$ such that there exists a morphism $h\colon F(A)\to F(C)$ with $dh=F(a)$. By Proposition \ref{prp:properties-t-theta} and applying $H$, we have $dH(t_a)=(-1)^{|a|}(a\circ H(t_A)-H(t_{C})\circ G(F(a)))$. Consequently, we get
		\[d((-1)^{|a|}a\circ H(\check t_A)+(-1)^{|a|}H(t_a\circ t'_{A})+H(t_{C})\circ G(h)\circ H(t'_{A}))=a .\]
		
		\item $H^*(F)\colon H^*(\cA)\to H^*(\cB)$ is full: Let $A,C\in\cA$ and $b\colon F(A)\to F(C)$ be a morphism with $db=0$. Define $y:=K(t_{F(C)})\circ F(H(t'_C))\circ b\circ F(H(t_A))\circ K(t'_{F(A)})\colon F(A)\to F(C)$, which satisfies $dy=0$. By Proposition \ref{prp:properties-t-theta} and applying $K$, we have $dK(t_y)=(-1)^{|y|}(y\circ K(t_{F(A)})-K(t_{F(C)})\circ F(G(y)))$. Consequently, we get $dh=F(H(t'_C))\circ b\circ F(H(t_A))-F(G(y))$ where
		\begin{multline*}
			h:=(-1)^{|y|}K(t'_{F(C)}\circ t_y)-K(\hat t_{F(C)})\circ F(G(y))
			+K(\hat t_{F(C)})\circ F(H(t'_C))\circ b\circ F(H(t_A))\circ K(t'_{F(A)}\circ t_{F(A)})\\+(-1)^{|b|}F(H(t'_C))\circ b\circ F(H(t_A))\circ K(\hat t_{F(A)}) .
		\end{multline*}
		Hence, we have $dk=b-F(H(t_C)\circ G(y)\circ H(t'_A))$ where
		\[k:=F(H(t_C))\circ h\circ F(H(t'_A))+F(H(\check t_C))\circ b\circ F(H(t_A\circ t'_A))+(-1)^{|b|}b\circ F(H(\check t_A)) .\]
	\end{itemize}
	These items prove that $F$ is a quasi-equivalence.
	
	Conversely, assume that $F\colon \cA\to\cB$ is a quasi-equivalence. We first define a dg functor $G\colon \cB\to\cA$ as follows: Assume $\{b_i\colon B_i\to D_i\}$ (indexed by an ordinal) is a set of generating morphisms of $\cB$.
	\begin{itemize}
		\item Let $B\in\cB$. Since $H^0(F)$ is essentially surjective, there exists $C\in\cA$ such that $F(C)$ is isomorphic to $B$ in $H^0(\cB)$. We set $G(B):=C$. Then, by Lemma \ref{lem:homotopy-equivalence}, there exist morphisms
		\[r_B\colon F(G(B))\to B,\, r'_B\colon B\to F(G(B)),\, \hat r_B\colon F(G(B))\to F(G(B)),\, \check r_B\colon B\to B,\, \bar r_B\colon F(G(B))\to B,\]
		where their gradings and differentials are given as in Lemma \ref{lem:homotopy-equivalence}. 
		
		\item Let $b_i\colon B_i\to D_i$ be a generating morphism such that $db_i=0$. Then, $r'_{D_i}\circ b_i\circ r_{B_i}\colon F(G(B_i))\to F(G(D_i))$ is a morphism with degree $|b_i|$ and zero differential. Since $H^*(F)$ is full, there exist morphisms
		$c_i\colon G(B_i)\to G(D_i)$ and $s_{b_i}\colon F(G(B_i))\to F(G(D_i))$
		such that $|c_i|=|b_i|$, $|s_{b_i}|=|b_i|-1$, $dc_i=0$ and $ds_{b_i}=r'_{D_i}\circ b_i\circ r_{B_i} - F(c_i)$. We set $G(b_i):=c_i$. Then, there exists a morphism $r_{b_i}\colon F(G(B_i))\to D_i$ such that $|r_{b_i}|=|b_i|-1$ and
		\[dr_{b_i}=(-1)^{|b_i|}(b_i\circ r_{B_i} - r_{D_i}\circ F(G(b_i))) .\]
		
		\item We proceed with transfinite induction: Assume $G$ is defined on morphisms generated by $\{b_j\colon B_j\to D_j\vb j<i\}$ as a dg functor, and assume that for $j<i$, there exist morphisms $r_{b_j}\colon F(G(B_j))\to D_j$ such that $|r_{b_j}|=|b_j|-1$ and
		\[dr_{b_j}=(-1)^{|b_j|}(b_j\circ r_{B_j} - r_{D_j}\circ F(G(b_j)))+r_{db_j} ,\]
		where $r_b$ for any morphism $b$ generated by $\{b_j\vb j<i\}$ is defined as in Definition \ref{dfn:generalised-t-morphisms} after replacing $t$ with $r$, $i_1$ with $F\circ G$, and $i_2$ with $1_{\cB}$. Also, Proposition \ref{prp:properties-t-theta} applies to the morphisms $r_b$ after such replacement. We want to define $G(b_i)$ and $r_{b_i}$.
		
		We have $d(r'_{D_i}\circ b_i\circ r_{B_i}+\hat r_{D_i}\circ F(G(db_i))+(-1)^{|b_i|}r'_{D_i}\circ r_{db_i}) = F(G(db_i))$. Since $H^*(F)$ is full and faithful, by Lemma \ref{lem:full-and-faithful}, there exist morphisms $c_i$ and $s_{b_i}$ such that $|c_i|=|b_i|$, $|s_{b_i}|=|b_i|-1$, $dc_i=G(db_i)$ and 
		$ds_{b_i}=r'_{D_i}\circ b_i\circ r_{B_i}+\hat r_{D_i}\circ F(G(db_i))+(-1)^{|b_i|}r'_{D_i}\circ r_{db_i}-F(c_i)$.
		We set $G(b_i):=c_i$. We also define
		\[r_{b_i}:=(-1)^{|b_i|} r_{D_i}\circ s_{b_i} + (-1)^{|b_i|} \check r_{D_i}\circ b_i\circ r_{B_i} + \check r_{D_i}\circ r_{db_i} - (-1)^{|b_i|} \bar r_{D_i}\circ F(G(db_i)),\]
		which satisfies $|r_{b_i}|=|b_i|-1$ and $dr_{b_i}=(-1)^{|b_i|}(b_i\circ r_{B_i} - r_{D_i}\circ F(G(b_i)))+r_{db_i}$. Note that we can also define $r_b$ for any morphisms $b\colon B\to D$ in $\cB$ in this case as in Definition \ref{dfn:generalised-t-morphisms} (after replacing $t$ with $r$, $i_1$ with $F\circ G$, and $i_2$ with $1_{\cB}$), which satisfies $|r_{b}|=|b|-1$ and $dr_{b}=(-1)^{|b|}(b\circ r_{B} - r_{D}\circ F(G(b)))+r_{db}$ by Proposition \ref{prp:properties-t-theta}.
	\end{itemize}
	Having defined the dg functor $G\colon \cB\to\cA$ and morphisms $r_{b}\colon F(G(B))\to D$, it is straightforward to check that $F\circ G\approx 1_{\cB}$ by constructing a dg functor $K\colon \Cyl(\cB)\to\cB$ satisfying $K(t^*_B)=r^*_B$ for each $B\in\cB$ and $t^*\in\{t, t',\hat t,\check t,\bar t\}$, and $K(t_{b_i})=r_{b_i}$ for each generating morphism $b_i$ of $\cB$. Finally, we want to show that $G\circ F\approx 1_{\cA}$. For that, we need to construct a dg functor $H\colon \Cyl(\cA)\to\cA$ such that $H\circ i_1=G\circ F$ and $H\circ i_2=1_{\cA}$. Hence, we only need to define $H(t^*_A)$ for any $A\in\cA$ satisfying $|H(t^*_A)|=|t^*_A|$ and $dH(t^*_A)=H(dt^*_A)$, and $H(t_{a_i})$ for any generating morphism $a_i$ of $\cA$ satisfying $|H(t_{a_i})|=|t_{a_i}|$ and $dH(t_{a_i})=H(dt_{a_i})$:
	\begin{itemize}
		\item Let $A\in\cA$. Since $H^*(F)$ is full, there exist morphisms
		\[u_A\colon G(F(A))\to A,\quad v_A\colon F(G(F(A)))\to F(A),\quad u'_A\colon A\to G(F(A)),\quad v'_A\colon F(A)\to F(G(F(A)))\]
		such that $dv_A=r_{F(A)}-F(u_A)$ and $dv'_A=r'_{F(A)}-F(u'_A)$. Next, since $H^*(F)$ is full and faithful and
		$d(\hat r_{F(A)}+r'_{F(A)}\circ v_A+v'_A\circ F(u_A))=F(1_{G(F(A))}-u'_A\circ u_A)$,
		by Lemma \ref{lem:full-and-faithful},
		there exist morphisms $\hat u_A$ and $\hat v_A$ such that $d\hat u_A=1_{G(F(A))}-u'_A\circ u_A$ and $d\hat v_A=\hat r_{F(A)}+r'_{F(A)}\circ v_A+v'_A\circ F(u_A) - F(\hat u_A)$. Similarly, we can define $\check u_A$ and $\check v_A$ satisfying $d\check u_A=1_A-u_A\circ u'_A$ and $d\check v_A=\check r_{F(A)}+v_A\circ r'_{F(A)}+F(u_A)\circ v'_A - F(\check u_A)$. Finally, by Lemma \ref{lem:full-and-faithful} again, we can show the existence of morphisms $\bar u_A$ and $\bar v_A$ satisfying $d\bar u_A=u_A\circ \hat u_A - \check u_A\circ u_A$ and $d\bar v_A= \bar r_{F(A)}+\check v_A\circ F(u_A)-F(u_A)\circ \hat v_A -\check r_{F(A)}\circ v_A- v_A\circ \hat r_{F(A)} -v_A\circ r'_{F(A)}\circ v_A- F(\bar u_A)$. Then, we set
		\[H(t_A):=u_A,\quad H(t'_A):=u'_A,\quad H(\hat t_A):=\hat u_A,\quad H(\check t_A):=\check u_A,\quad H(\bar t_A):=\bar u_A .\]
		
		\item Let $a_i\colon A_i\to C_i$ be a generating morphism of $\cA$ satisfying $da_i=0$. First, observe that $d(r_{F(a_i)}-F(a_i)\circ v_{A_i}+(-1)^{|a_i|}v_{C_i}\circ F(G(F(a_i))))=F((-1)^{|a_i|}(a_i\circ u_{A_i}-u_{C_i}\circ G(F(a_i))))$. Hence, by Lemma \ref{lem:full-and-faithful}, there exist $u_{a_i}$ and $v_{a_i}$ satisfying
		\[du_{a_i}=(-1)^{|a_i|}(a_i\circ u_{A_i}-u_{C_i}\circ G(F(a_i)))=H(dt_{a_i})\]
		and $dv_{a_i}=r_{F(a_i)}-F(a_i)\circ v_{A_i}+(-1)^{|a_i|}v_{C_i}\circ F(G(F(a_i))) - F(u_{a_i})$. We set $H(t_{a_i}):=u_{a_i}$.
		
		\item We proceed with transfinite induction: Let $a_i\colon A_i\to C_i$ be a generating morphism of $\cA$. Assume $H(t_{a_j}):=u_{a_j}$ and $v_{a_j}$ are defined for generating morphisms $\{a_j\colon A_j\to C_j\vb j<i\}$ with $du_{a_j}=(-1)^{|a_j|}(a_j\circ u_{A_j}-u_{C_j}\circ G(F(a_j)))+u_{da_j}$ and 
		\[dv_{a_j}=r_{F(a_j)}-F(a_j)\circ v_{A_j}+(-1)^{|a_j|}v_{C_j}\circ F(G(F(a_j))) - v_{da_j} - F(u_{a_j}).\]
		Here, we set $u_a:=H(t_a)$ for any morphism $a$ generated by $\{a_j\vb j<i\}$, and
		for any morphism $a=a_{j_n}\circ\ldots\circ a_{j_1}$ with $a_{j_l}\in\{a_j\vb j<i\}$, we define
		\[v_a:=\sum_{l=1}^n (-1)^{|a|-|a_l|}F(a_{j_n}\circ\ldots\circ a_{j_{l+1}})\circ v_{a_{j_l}}\circ F(G(F(a_{j_{l-1}}\circ\ldots \circ a_{j_1}))) .\]
		The definition extends linearly to define $v_a$ for any morphism $a$ generated by $\{a_j\vb j<i\}$. Then, a version of Proposition \ref{prp:properties-t-theta} holds for the morphisms $v_a$. In particular, we have
		\[dv_a=r_{F(a)}-F(a)\circ v_{A}+(-1)^{|a|}v_{C}\circ F(G(F(a)))- v_{da}- F(u_{a}) \]
		for any morphism $a\colon A\to C$ generated by $\{a_j\vb j<i\}$. We want to show that there exist morphisms $u_{a_i}$ and $v_{a_i}$ satisfying $du_{a_i}=(-1)^{|a_i|}(a_i\circ u_{A_i}-u_{C_i}\circ G(F(a_i)))+u_{da_i}$ and 
		\[dv_{a_i}=r_{F(a_i)}-F(a_i)\circ v_{A_i}+(-1)^{|a_i|}v_{C_i}\circ F(G(F(a_i)))-v_{da_i}- F(u_{a_i}) .\]
		We observe that
		\[d(r_{F(a_i)}-F(a_i)\circ v_{A_i}+(-1)^{|a_i|}v_{C_i}\circ F(G(F(a_i))) - v_{da_i})=F((-1)^{|a_i|}(a_i\circ u_{A_i}-u_{C_i}\circ G(F(a_i)))+u_{da_i}) .\]
		Hence, by Lemma \ref{lem:full-and-faithful}, the desired morphisms $u_{a_i}$ and $v_{a_i}$ exist. We set $H(t_{a_i}):=u_{a_i}$.
	\end{itemize}
	These items define the dg functor $H\colon\Cyl(\cA)\to\cA$ satisfying $H\circ i_1=G\circ F$ and $H\circ i_2=1_{\cA}$, which implies $G\circ F\approx 1_{\cA}$. Hence, $F$ is a homotopy equivalence.
\end{proof}

At this stage, we can provide a computable functorial factorizations of dg functors as a composition of a cofibration and a quasi-equivalence. We will achieve this using mapping cylinders of dg functors, which we will define below:

\begin{dfn}\label{dfn:mapping-cylinder-dg}
	Let $\cA$ be a semifree dg category and $F\colon\cA\to\cB$ be a dg functor.
	\begin{enumerate}
		\item We define the {\em mapping cylinder} $M_F$ of $F$ as the pushout
		\[M_F:=\colim(\Cyl(\cA)\xleftarrow{i_1}\cA\xrightarrow{F}\cB) .\]
		Explicitly, $M_F$ is given as the semifree extension of $\cB\amalg\cA$ by the morphisms
		\[t_A\colon F(A)\to A,\quad t'_A\colon A\to F(A), \quad \hat t_A\colon F(A)\to F(A),\quad \check t_A\colon A\to A,\quad \bar t_A\colon F(A)\to A\]
		for each $A\in\cA$, and the morphism
		\[t_{a_i}\colon F(X_i)\to Y_i\]
		for each generating morphism $a_i\colon X_i\to Y_i$ of $\cA$, with the gradings
		\[|t_A|=|t'_A|=0,\quad |\hat t_A|=|\check t_A|=-1,\quad |\bar t_A|=-2,\quad |t_{a_i}|=|a_i|-1,\]
		and differentials
		\begin{gather*}
			dt_A=dt'_A=0,\quad d\hat t_A=1_{F(A)}-t'_A\circ t_A,\quad d\check t_A=1_{A}-t_A\circ 	t'_A,\quad d\bar t_A=t_A\circ \hat t_A-\check t_A\circ t_A,\\
			dt_{a_i}=(-1)^{|a_i|}(a_i\circ t_{X_i}-t_{Y_i}\circ F(a_i)) +t_{da_i} ,
		\end{gather*}
		where for any morphism $a$ in $\cA$, $t_a$ is defined as in Definition \ref{dfn:generalised-t-morphisms} (after replacing $i_1$ with $F$ and $i_2$ with $1_{\cA}$).
		
		\item For any commutative diagram
		\[\begin{tikzcd}
			\cA\ar[r,"F"]\ar[d,"\alpha",blue] & \cB\ar[d,"\beta",blue]\\
			\cA'\ar[r,"F'"] & \cB'
		\end{tikzcd}\]
		in $\sCat$, we define a dg functor $M(\alpha,\beta)\colon M_F\to M_{F'}$ as the image of the morphism of the diagrams
		\[\begin{tikzcd}
			\Cyl(\cA)\ar[d,"\Cyl(\alpha)",blue] & \cA\ar[l,"i_1"']\ar[r,"F"]\ar[d,"\alpha",blue] & \cB\ar[d,"\beta",blue]\\
			\Cyl(\cA')& \cA'\ar[l,"i_1"'] \ar[r,"F'"] & \cB'
		\end{tikzcd}\]
		under the colimit functor. Explicitly, $M(\alpha,\beta)$ is given by
		\[M(\alpha,\beta)|_{\cB\amalg\cA}=\beta\amalg\alpha,\quad M(\alpha,\beta)(t^*_A)=t^*_{\alpha(A)},\quad M(\alpha,\beta)(t_{a_i})=t_{\alpha(a_i)},\]
		where $t^*\in\{t,t',\hat t,\check t, \bar t\}$.
	\end{enumerate}
\end{dfn}

\begin{thm}\label{thm:mapping-cylinder-factorisation}
	Let $F\colon \cA\to\cB$ be a dg functor between semifree dg categories, $M_F$ be the mapping cylinder of $F$, and $M(\alpha,\beta)\colon M_F\to M_{F'}$ be a dg functor as in Definition \ref{dfn:mapping-cylinder-dg}.
	\begin{enumerate}
		\item\label{item:mapping-obj} There exists a factorization of $F$ by dg functors
		\[F\colon \cA\xrightarrow{j} M_F\xrightarrow{q} \cB ,\]
		where $j$ is defined as the composition $\cA\hookrightarrow \cB\amalg\cA\hookrightarrow M_F$, and $q$ is defined as
		\[q|_{\cB\amalg\cA}:=1_{\cB}\amalg F,\quad q(t_A)=q(t'_A):=1_{F(A)},\quad q(\hat t_A)=q(\check t_A)=q(\bar t_A)=q(t_{a_i}):=0 .\]
		Moreover, $j$ is a cofibration and $q$ is a quasi-equivalence.
		
		\item\label{item:mapping-mor} The factorization above is functorial. In other words, for every commutative square
		\[\begin{tikzcd}
			\cA\ar[r,"F"]\ar[d,"\alpha",blue] & \cB\ar[d,"\beta",blue]\\
			\cA'\ar[r,"F'"] & \cB'
		\end{tikzcd}\]
		in $\sCat$, the diagram
		\[\begin{tikzcd}
			\cA\ar[r,"j"]\ar[d,"\alpha",blue] & M_F\ar[r,"q"]\ar[d,"{M(\alpha,\beta)}",blue] & \cB\ar[d,"\beta",blue]\\
			\cA'\ar[r,"j'"] & M_{F'}\ar[r,"q'"] & \cB'
		\end{tikzcd}\]
		commutes, $M(1_{\cA},1_{\cA})=1_{M_{F}}$, and $M(\alpha',\beta')\circ M(\alpha,\beta)=M(\alpha'\circ\alpha,\beta'\circ\beta)$.
	\end{enumerate}
\end{thm}

\begin{proof}
	Since $\dgCat_s$ is an $I$-category with $I=\Cyl$ by Theorem \ref{thm:dgcat-I-category}, \cite[Lemma 3.12]{baues} shows that the first item holds with the distinction that $q$ is a homotopy equivalence. Then, Theorem \ref{thm:whitehead-for-semifree} implies that $q$ is a quasi-equivalence. As for the second item, it is straightforward to check.
\end{proof}

\begin{rmk}
	It is easy to check that the dg functor $q$ above is a fibration in the Dwyer-Kan model category of dg categories given in Theorem \ref{thm:dgcat-model}. See e.g. \cite{hocolim} for the definition of fibrations. Hence, Theorem \ref{thm:mapping-cylinder-factorisation} gives an explicit factorization of $F$ into a cofibration and an acyclic fibration in the Dwyer-Kan model category. This factorization also holds in quasi-equiconic and Morita model categories of dg categories given in Theorem \ref{thm:dgcat-model}, as they are left Bousfield localization of the Dwyer-Kan model category.
\end{rmk}

We are ready to present one of the main results of this section:

\begin{thm}\label{thm:dgcat-cofibration-category}
	The category of semifree dg categories $\dgCat_s$ forms a cofibration category, where weak equivalences are quasi-equivalences and cofibrations are semifree extensions. Moreover, every object in $\dgCat_s$ is both fibrant and cofibrant. In particular, $\dgCat_s$ makes a category of cofibrant objects.
\end{thm}

\begin{proof}[Proof of Theorem \ref{thm:dgcat-cofibration-category}]
	According to \cite{baues}, every $I$-category is a cofibration category with the same cofibrations, and its weak equivalences are homotopy equivalences. Additionally, each of its objects is both fibrant and cofibrant. Since homotopy equivalences between semifree dg categories are exactly quasi-equivalences by Theorem \ref{thm:whitehead-for-semifree}, the result follows. Finally, cofibrant objects in a cofibration category makes a category with cofibrant objects, see e.g. \cite{kamps-porter} for the definition.
\end{proof}

\begin{rmk}\label{rmk:hocolim-from-cofibration}
	Since Theorem \ref{thm:dgcat-cofibration-category} realizes $\sCat$ as a cofibration category, we can consider homotopical constructions such as homotopy pushouts in $\sCat$, see \cite{baues}. Specifically, we can define the homotopy pushout as
	\[\hocolim(\cA\xleftarrow{\alpha}\cC\xrightarrow{\beta}\cB):=
	\colim(M_{\alpha}\xleftarrow{j}\cC\xrightarrow{\beta}\cB)\]
	up to quasi-equivalence using Theorem \ref{thm:mapping-cylinder-factorisation}\eqref{item:mapping-obj}. As a consequence, we can directly obtain Theorem \ref{thm:hocolim-functor-dg} for $\sCat$ using Theorem \ref{thm:mapping-cylinder-factorisation}\eqref{item:mapping-obj} and Theorem \ref{thm:mapping-cylinder-factorisation}\eqref{item:mapping-mor}.
\end{rmk}

We conclude this section by stating the main theorems for semifree dg algebras and dg categories/algebras of strictly finite type. Their proofs follow a similar approach to those for semifree dg categories.

\begin{thm}\label{thm:i-cof-cat-dg-algebras}\mbox{}
	\begin{enumerate}
		\item The category of semifree dg algebras $\textup{dgAlg}_s$ is an $I$-category with the structure $(\mathrm{cof},I)$, which is defined as follows:
		\begin{itemize}
			\item $\mathrm{cof}:$ Cofibrations are semifree extensions.
			
			\item $I$: The functor $I$ is the cylinder functor $\Cyl\colon\textup{dgAlg}_s \to\textup{dgAlg}_s$ described in Remark \ref{rmk:cylinder-functor-dg-algebras}.
		\end{itemize}
		
		\item The category of semifree dg algebras $\textup{dgAlg}_s$ forms a cofibration category, where weak equivalences are quasi-equivalences and cofibrations are semifree extensions. Moreover, every object in $\textup{dgAlg}_s$ is both fibrant and cofibrant. In particular, $\textup{dgAlg}_s$ makes a category of cofibrant objects.
	\end{enumerate}
\end{thm}

\begin{rmk}\label{rmk:finite-type}
	Let $\dgCat_f$ be the category of dg categories of strictly finite type, i.e., semifree dg categories with finitely many objects and generating morphisms. It is easy to check that Theorem \ref{thm:dgcat-I-category} and \ref{thm:dgcat-cofibration-category} hold if we replace $\dgCat_s$ with $\dgCat_f$. The same is true if we consider the category $\dgAlg_f$ of dg algebras of strictly finite type, i.e., semifree dg algebras with finitely many generating morphisms: Theorem \ref{thm:i-cof-cat-dg-algebras} still holds if we replace $\dgAlg_s$ with $\dgAlg_f$.
\end{rmk}

\section{Wrapped Fukaya category of $T^*S^n$ and the reflection functor}\label{sec:sphere-reflection}

In this section, we give a simple application of Theorem \ref{thm:hocolim-functor-dg} and \ref{thm:hocolim-functor-dg-loc} within the context of symplectic geometry.

Consider two symplectic manifolds $W_1$ and $W_2$, and a symplectic map $f: W_1 \to W_2$, which is a smooth map respecting their symplectic structures. 
It is well-known that $f$ induces a functor between their Fukaya categories. 
It is a natural question to ask how to construct the induced functor from $f$. 

For a Weinstein manifold, i.e., an open symplectic manifold satisfying certain conditions, \cite{gps2} proved that the corresponding wrapped Fukaya category can be computed, after a choice of a covering, as a homotopy colimit of the wrapped Fukaya categories of the covering elements.
Furthermore, if $f$ respects the coverings, applying Theorems \ref{thm:hocolim-functor-dg} or \ref{thm:hocolim-functor-dg-loc} allows us to describe the functor induced by $f$. To illustrate this, we provide an example of constructing an induced functor using Theorems \ref{thm:hocolim-functor-dg} and \ref{thm:hocolim-functor-dg-loc}. We consider the symplectic manifold $T^*S^n$, where $S^n$ is the $n$-dimensional sphere.

Before going further, let us discuss motivations behind considering these induced functors.

The first motivation stems from a conjecture written in \cite{Kontsevich09}, which questions whether the group of autofunctors of the Fukaya category coincides with the (stabilized) group of symplectic automorphisms. If this conjecture holds true, the induced functor would encapsulate the same information as the original symplectic automorphism. Thus, studying induced functors becomes instrumental in exploring symplectic automorphisms.

The second motivation is the homotopy colimit computation of the wrapped Fukaya category, established in \cite{gps2}. This computation involves a homotopy colimit diagram comprising categories and functors between them. For instance, in Theorems \ref{thm:hocolim-functor-dg} and \ref{thm:hocolim-functor-dg-loc}, the homotopy colimit diagram includes functors denoted as $\alpha$ and $\beta$. Therefore, to fully describe a homotopy colimit diagram and subsequently compute the wrapped Fukaya category via a homotopy colimit, it is essential to specify these corresponding functors.

In Section \ref{sec:preliminaries-symplectic}, we recall some facts about the wrapped Fukaya category of Weinstein manifolds. In Section \ref{sec:wfuk-sphere}, we calculate the wrapped Fukaya category $\cW(T^*S^n)$ of the cotangent bundle $T^*S^n$. In Section \ref{sec:reflection}, we present the dg functor $R_n\colon\cW(T^*S^n)\to\cW(T^*S^n)$ that is induced from the reflection map $r_n\colon T^*S^n\to T^*S^n$ along an axis.

\subsection{Preliminaries on wrapped Fukaya categories}\label{sec:preliminaries-symplectic}

In this subsection, we will review some key aspects of a powerful invariant known as the ``wrapped Fukaya category'' $\mathcal{W}(W)$ of a Weinstein manifold/sector $W$. A {\em Weinstein manifold} is an exact symplectic manifold whose Liouville vector field is both complete and gradient-like for an exhausting Morse function. A {\em Weinstein sector} is a Weinstein manifold with boundary, where its Liouville vector field satisfies certain conditions near the boundary. For a more detailed exposition, readers can consult \cite{weinstein} and \cite{gps1}.

The {\em wrapped Fukaya category} $\cW(W)$ of $W$ is an $A_{\infty}$-category whose objects are certain exact Lagrangians in $W$ with cylindrical ends, equipped with additional data. Morphisms in this category are generated by the intersections of Lagrangians after perturbing them through a process known as ``wrapping''. $A_{\infty}$-operations arise from counting pseudoholomorphic polygons bounded by Lagrangians (edges) and their intersections (corners). For a rigorous definition, refer to \cite{gps1} and \cite{seidel}.

The subsequent subsections will focus on the case where $W$ is a cotangent bundle, specifically, $W=T^*S^n$ for some $n$.

\begin{rmk}\label{rmk:wrapped-fukaya}\mbox{}
	\begin{enumerate}
		\item Any $A_{\infty}$-category $\cC$ can be regarded as a dg category up to quasi-equivalence by replacing $\cC$ with its image under the $A_{\infty}$-Yoneda embedding. Hence, wrapped Fukaya categories can be regarded as dg categories up to quasi-equivalence.
		
		\item\label{item:grading-pin} The wrapped Fukaya category $\cW(W)$ can be given $\Z$-grading if $2c_1(W)=0\in H^2(W;\Z)$. Also, the definition of $\cW(W)$ depends on the classes $\eta\in H^1(W;\Z)$ (grading structure) and $b\in H^2(W;\Z/2)$ (background class), which are used to give gradings on the Lagrangian intersections and orientations of moduli spaces of pseudoholomorphic disks, respectively. See \cite{seidel} for more details, or \cite{cat-entropy} for a quick overview.
	\end{enumerate}
\end{rmk}

The {\em skeleton} of a Weinstein manifold/sector $W$ can be defined as the union of the stable manifolds of the zeroes of the Liouville vector field of $W$. Then, we have the following theorem:

\begin{thm}[\cite{CDRGG,gps2}]\label{thm:cocore-generate}
	Let $W$ be a Weinstein manifold (or sector) of dimension $2n$. Then $\cW(W)$ is generated by {\em Lagrangian cocores}, which are Lagrangian disks dual to $n$ dimensional strata of the skeleton of $W$.
\end{thm}

\begin{rmk}\label{rmk:cotangent-generation}
	When $W=T^*M$ for a smooth manifold $M$, its skeleton is the zero section $M$, and Lagrangian cocores correspond to cotangent fibers. Then, Theorem \ref{thm:cocore-generate} implies that the wrapped Fukaya category $\cW(T^*M)$ is generated by cotangent fibers of $T^*M$, which is originally due to \cite{abouzaid-wrapped-generation}. Moreover, if $\cW(T^*M)$ is equipped with the standard grading structure and the background class as described in \cite{nadler-zaslow}, \cite{Abouzaid12} (or \cite[Example 1.36]{gps2}) shows that
	\[\hom^*(L_p,L_p)\simeq C_{-*}(\Omega_p(M))\]
	where $L_p$ is the cotangent fiber at $p\in M$, and $C_{-*}(\Omega_p(M))$ is the normalized cubical chains on the based (Moore) loop space $\Omega_p(M)$ of $M$ at $p$. The product structure on $C_{-*}(\Omega_p(M))$ is induced by the concatenation of Moore loops, which is strictly associative.
\end{rmk}

Given an inclusion of Weinstein sectors $F\hookrightarrow W$, there is an induced $A_{\infty}$-functor $\cW(F)\to\cW(W)$ as described in \cite{gps1}. Then, we state the following theorem, which will serve as our primary tool for computing wrapped Fukaya categories:

\begin{thm}[\cite{gps2}]\label{thm:gps}
	Let $W$ be a Weinstein manifold (or sector). Suppose $W=W_1\cup W_2$ for some Weinstein sectors $W_1$ and $W_2$, and $W_1\cap W_2$ is a hypersurface in $W$ whose neighborhood can be written as $F\times T^*[0,1]$, where $F$ is Weinstein (up to deformation). Then there is a pretriangulated equivalence
	\[\cW(W)\simeq\hocolim\left(
	\begin{tikzcd}
		\cW(W_1) & & \cW(W_2)\\
		& \cW(F)\ar[lu]\ar[ru]
	\end{tikzcd}
	\right) ,\]
	where the arrows are induced by the inclusion of Weinstein sectors $F\hookrightarrow W_i$ for $i=1,2$.
\end{thm}

\subsection{Wrapped Fukaya category $\cW(T^*S^n)$ of $T^*S^n$}\label{sec:wfuk-sphere}

First, note that by Remark \ref{rmk:wrapped-fukaya}\eqref{item:grading-pin}, the wrapped Fukaya category $\cW(T^*S^n)$ can be given $\Z$-grading since $2c_1(T^*S^n)=0\in H^2(T^*S^n;\Z)$ for any $n\geq 1$. Also by Remark \ref{rmk:wrapped-fukaya}\eqref{item:grading-pin}, the definition of $\cW(T^*S^n)$ depends on the grading structure $\eta\in H^1(T^*S^n;\Z)$ and the background class $b\in H^2(T^*S^n;\Z/2)$. They are uniquely determined for $T^*S^n$ for $n\geq 3$. For $T^*S^1$ and $T^*S^2$, we choose the standard grading structure and the background class as in \cite{nadler-zaslow}. We will comment on the nonstandard choices in Remark \ref{rmk:wfuk-sphere-nonstandard}.

\begin{prp}\label{prp:wfuk-sphere}
	Let $n\geq 1$. The wrapped Fukaya category of $T^*S^n$ is given, up to pretriangulated equivalence, by
	\[\cW(T^*S^n)\simeq 
		\begin{cases}
			\cC_1[z^{-1}] &\text{if }n=1\\
			\cC_n &\text{if }n\geq 2
		\end{cases}\]
	where $\cC_n$ is the semifree dg category given as follows:
	\begin{enumerate}[label = (\roman*)]
		\item {\em Objects:} $L$ (representing a cotangent fiber of $T^*S^n$).
		\item {\em Generating morphisms:} $z\in\hom^*(L,L)$.
		\item {\em Degrees:} $|z|=1-n$.
		\item {\em Differentials:} $dz=0$.
	\end{enumerate}
\end{prp}

\begin{rmk}\label{rmk:wfuk-sphere-nonstandard}\mbox{}
	\begin{enumerate}
		\item For $T^*S^1$, there are $\Z$-many grading structures to define $\cW(T^*S^1)$. To capture the nonstandard ones, one needs to let $|z|=m$ for $m\in\Z$ instead of $|z|=0$.
		\item For $T^*S^2$, there is another background class to define $\cW(T^*S^2)$. To capture the nonstandard one, one needs to replace $dz=1_L-1_L=0$ with $dz=1_L+1_L=2\cdot 1_L$. 
	\end{enumerate}
	
\end{rmk}

\begin{proof}[Proof of Proposition \ref{prp:wfuk-sphere}]
	The cases $n=1$ and $n=2$ are proved in e.g. \cite{pinwheel} using microlocal sheaves, but we will prove it independently for any $n\geq 1$.  By Theorem \ref{thm:gps}, we have the pretriangulated equivalence
	\begin{align}\label{eq:sphere-hocolim}
		\cW(T^*S^n)&\simeq\hocolim(\cW(T^*D^n)\leftarrow \cW(T^*S^{n-1})\rightarrow \cW(T^*D^n))\notag\\
		&\simeq\hocolim(\Tw(\cA_1(1))\leftarrow \cW(T^*S^{n-1})\rightarrow \Tw(\cA_1(2)))
	\end{align}
	where we used the pretriangulated equivalence $\cW(T^*D^n)\simeq\cA_1(i)$ for any $i$, where
	$\cA_1(i)$ is the semifree dg category with a single object $K_i$ (representing a cotangent fiber) and no generating morphisms. Note that we use $i$ just to distinguish multiple $\cA_1$'s.
	
	\textbf{Case $n=1$:} We have the pretriangulated equivalence $\cW(T^*S^0)\simeq \cA_1(3)\amalg\cA_1(4)$ because $S^0$ is a disjoint union of two points, hence (\ref{eq:sphere-hocolim}) becomes
	\begin{align*}
		\cW(T^*S^1)&\simeq\hocolim(\Tw(\cA_1(1))\leftarrow \Tw(\cA_1(3)\amalg\cA_1(4))\rightarrow \Tw(\cA_1(2)))\\
		&\simeq\hocolim(\cA_1(1)\leftarrow \cA_1(3)\amalg\cA_1(4)\rightarrow \cA_1(2)) .
	\end{align*}
	The second pretriangulated equivalence makes sense since the functors $\Tw(\cA_1(3)\amalg\cA_1(4))\rightarrow \Tw(\cA_1(i))$ for $i=1,2$ are induced by the dg functors
	\begin{align*}
		\cA_1(3)\amalg\cA_1(4)&\rightarrow \cA_1(i)\\
		K_3, K_4&\mapsto K_i .
	\end{align*}
	Then \cite{hocolim} (or Theorem \ref{thm:hocolim-functor-dg}\eqref{item:hocolim-obj}) gives the pretriangulated equivalence
	\[\cW(T^*S^1)\simeq\cC_1'[\{t_{K_3},t_{K_4}\}^{-1}]\]
	where $\cC_1'$ is the semifree dg category given as follows:
	\begin{enumerate}[label = (\roman*)]
		\item {\em Objects:} $K_1,K_2$.
		\item {\em Generating morphisms:} $t_{K_3},t_{K_4}\in\hom^*(K_1,K_2)$.
		\item {\em Degrees:} $|t_{K_3}|=|t_{K_4}|=0$.
		\item {\em Differentials:} $dt_{K_3}=dt_{K_4}=0$.
	\end{enumerate}
	Since $K_1$ and $K_2$ are homotopy equivalent in $\cC_1'[t_{K_4}^{-1}]$, we have the quasi-equivalence
	\begin{align*}
		\cC_1'[t_{K_4}^{-1}]&\xrightarrow{\sim}\cC_1\\
		K_1,K_2\mapsto L,\quad	t_{K_3}&\mapsto z,\quad t_{K_4}\mapsto 1_L
	\end{align*}
	hence we have the quasi-equivalence
	\[\cC_1'[\{t_{K_3},t_{K_4}\}^{-1}]\simeq \cC_1[z^{-1}]\]
	which then gives the pretriangulated equivalence
	\[\cW(T^*S^1)\simeq\cC_1[z^{-1}] .\]
	
	\textbf{Case $n=2$:} By the previous case (by renaming $z$ by $x$ for notational purposes), (\ref{eq:sphere-hocolim}) becomes
	\begin{align*}
		\cW(T^*S^2)&\simeq\hocolim(\Tw(\cA_1(1))\leftarrow \Tw(\cC_1[x^{-1}])\rightarrow \Tw(\cA_1(2)))\\
		&\simeq\hocolim(\cA_1(1)\leftarrow \cC_1[x^{-1}]\rightarrow \cA_1(2)) .
	\end{align*}
	The second pretriangulated equivalence makes sense since the functors $\Tw(\cC_1[x^{-1}])\rightarrow \Tw(\cA_1(i))$ for $i=1,2$ are induced by the dg functors
	\begin{align*}
		\cC_1[x^{-1}]&\rightarrow \cA_1(i)\\
		L&\mapsto K_i,\quad x\mapsto 1_{K_i} .
	\end{align*}
	Then \cite{hocolim} (or Theorem \ref{thm:hocolim-functor-dg-loc}\eqref{item:hocolim-loc-obj}) gives the pretriangulated equivalence
	\[\cW(T^*S^2)\simeq\cC_2'[t_L^{-1}]\]
	where $\cC_2'$ is the semifree dg category given as follows:
	\begin{enumerate}[label = (\roman*)]
		\item {\em Objects:} $K_1,K_2$.
		\item {\em Generating morphisms:} $t_{L},t_x\in\hom^*(K_1,K_2)$.
		\item {\em Degrees:} $|t_L|=0,\quad |t_x|=-1$.
		\item {\em Differentials:} $dt_{L}=0,\quad dt_x=t_L-t_L=0$.
	\end{enumerate}
	Since $K_1$ and $K_2$ are homotopy equivalent in $\cC_2'[t_L^{-1}]$, we have the quasi-equivalence
	\begin{align*}
		\cC_2'[t_L^{-1}]&\xrightarrow{\sim}\cC_2\\
		K_1,K_2\mapsto L,\quad t_L&\mapsto 1_L,\quad t_x\mapsto z
	\end{align*}
	which then gives the pretriangulated equivalence
	\[\cW(T^*S^2)\simeq\cC_2 .\]
	
	\textbf{Case $n\geq 3$:} Assume the proposition holds for $n-1$, i.e. we have the pretriangulated equivalence $\cW(T^*S^{n-1})\simeq \cC_{n-1}$ (rename $z$ in $\cC_{n-1}$ by $x$ for notational purposes). Then (\ref{eq:sphere-hocolim}) becomes
	\begin{align*}
		\cW(T^*S^n)&\simeq\hocolim(\Tw(\cA_1(1))\leftarrow \Tw(\cC_{n-1})\rightarrow \Tw(\cA_1(2)))\\
		&\simeq\hocolim(\cA_1(1)\leftarrow \cC_{n-1}\rightarrow \cA_1(2)) .
	\end{align*}
	The second pretriangulated equivalence makes sense since the functors $\Tw(\cC_{n-1})\rightarrow \Tw(\cA_1(i))$ for $i=1,2$ are induced by the dg functors
	\begin{align*}
		\cC_{n-1}&\rightarrow \cA_1(i)\\
		L&\mapsto K_i,\quad x\mapsto 0 .
	\end{align*}
	Then \cite{hocolim} (or Theorem \ref{thm:hocolim-functor-dg}\eqref{item:hocolim-obj}) gives the pretriangulated equivalence
	\[\cW(T^*S^n)\simeq\cC_n'[t_L^{-1}]\]
	where $\cC_n'$ is the semifree dg category given as follows:
	\begin{enumerate}[label = (\roman*)]
		\item {\em Objects:} $K_1,K_2$.
		\item {\em Generating morphisms:} $t_{L},t_x\in\hom^*(K_1,K_2)$.
		\item {\em Degrees:} $|t_L|=0,\quad |t_x|=1-n$.
		\item {\em Differentials:} $dt_{L}=0,\quad dt_x=0$.
	\end{enumerate}
 	Since $K_1$ and $K_2$ are homotopy equivalent in $\cC_n'[t_L^{-1}]$, we have the quasi-equivalence
	\begin{align*}
		\cC_n'[t_L^{-1}]&\xrightarrow{\sim}\cC_n\\
		K_1,K_2\mapsto L,\quad t_L&\mapsto 1_L,\quad t_x\mapsto z
	\end{align*}
	which then gives the pretriangulated equivalence
	\[\cW(T^*S^n)\simeq\cC_n .\]
\end{proof}

\subsection{The reflection functor on $\cW(T^*S^n)$}\label{sec:reflection}

There is a reflection map
\begin{align*}
	r_n\colon T^*\R^{n+1} &\to T^*\R^{n+1}\\
	(x_1,x_2,\ldots,x_{n+1},y_1,y_2,\ldots,y_{n+1})&\mapsto (-x_1,x_2,\ldots,x_{n+1},-y_1,y_2,\ldots,y_{n+1})
\end{align*}
where $x_i$ are the base coordinates and $y_i$ are fiber coordinates. By considering its restriction to $T^*S^n$ via $\sum_{i=1}^n x_i^2=1$ and $\sum_{i=1}^n x_i y_i=0$, we get an exact symplectomorphism
\[r_n\colon T^*S^n \to T^*S^n .\]
Since $r_n$ is an exact symplectomorphism, it induces an $A_{\infty}$-quasi-equivalence
\[R_n\colon \cW(T^*S^n)\to\cW(T^*S^n) .\]
For convenience, we will use the term ``the reflection map/functor" to denote the exact symplectomorphism $r_n$ and its induced functor $R_n$.
Our goal in the subsection is to understand the reflection functor $R_n$ explicitly using the pretriangulated equivalences $\cW(T^*S^1)\simeq\cC_1[z^{-1}]$ and $\cW(T^*S^n)\simeq\cC_n$ for $n\geq 2$ given by Proposition \ref{prp:wfuk-sphere}. We will use Theorem \ref{thm:hocolim-functor-dg} and $\ref{thm:hocolim-functor-dg-loc}$ to achieve this. The result is given by the following proposition:

\begin{prp}\label{prp:reflection}
	The reflection map $r_n\colon T^*S^n\to T^*S^n$ induces a dg functor (up to $A_{\infty}$-natural equivalence)
	\[R_n\colon \cW(T^*S^n)\to\cW(T^*S^n)\]
	which is given as follows:
	
	\begin{enumerate}
		\item If $n=1$, then
			\begin{align*}
				R_1\colon \Tw(\cC_1[z^{-1}])&\to\Tw(\cC_1[z^{-1}])\\
				L &\mapsto L\\
				z& \mapsto z'
			\end{align*}
			where $\cC_1$ is the semifree dg category defined in Proposition \ref{prp:wfuk-sphere}, and $z'$ is the inverse of $z$ up to homotopy.
		
		\item If $n\geq 2$, then
			\begin{align*}
				R_n\colon \Tw(\cC_n)&\to\Tw(\cC_n)\\
				L &\mapsto L\\
				z& \mapsto -z
			\end{align*}
			where $\cC_n$ is the semifree dg category defined in Proposition \ref{prp:wfuk-sphere}. 
	\end{enumerate}
\end{prp}

\begin{proof}
	The reflection map $r_n\colon T^*S^n\to T^*S^n$ sends a cotangent fiber to a cotangent fiber, hence $R_n(L)=L[m]$ for some $m\in\Z$. Since $r_n\circ r_n=1$, we have $R_n\circ R_n\simeq 1$, and consequently, we must have $m=0$. Hence, $R_n(L)=L$.
	
	Since $L$ generates $\cW(T^*S^n)$, $R_n$ is determined by $R_n\colon\hom^*(L,L)\to\hom^*(L,L)$, i.e., we need to understand the $A_{\infty}$-functors
	\[R_1\colon\cC_1[z^{-1}]\to\cC_1[z^{-1}]\quad\text{and}\quad R_n\colon \cC_n\to \cC_n \text{ for $n\geq 2$}\]
	by Proposition \ref{prp:wfuk-sphere}.
	In fact, we can regard $R_n$ as a dg functor, since $\cC_n$ is a semifree dg category (in particular, cofibrant).
	
	To determine $R_n$, first divide $S^n$ with the hyperplane $\{x_{n+1}=0\}$ into two $n$-dimensional closed disks $D^n$ whose intersection is $S^{n-1}$. This induces the gluing diagram of Weinstein sectors for $T^*S^n$
	\[T^*D^n\leftarrow T^*S^{n-1}\rightarrow T^*D^n .\]
	The reflection map $r_n\colon T^*S^n\to T^*S^n$ can be restricted to the gluing diagram
	such that the maps $(r_n)_{T^*D^n}\colon T^*D^n\to T^*D^n$ and $(r_n)_{T^*S^{n-1}}=r_{n-1}\colon T^*S^{n-1}\to T^*S^{n-1}$ are reflection maps, and
	the following diagram commutes:
	\[\begin{tikzcd}
		T^*D^n\ar[d,"(r_n)_{T^*D^n}", blue] & T^*S^{n-1}\ar[l]\ar[r]\ar[d,"r_{n-1}", blue] & T^*D^n\ar[d,"(r_n)_{T^*D^n}", blue]\\
		T^*D^n & S^{n-1}\ar[l]\ar[r] & T^*D^n
	\end{tikzcd}\]
	The induced functor by $(r_n)_{T^*D^n}$ is clearly identity on $\cW(T^*D^n)$. Hence, the crucial data is the induced functor $R_{n-1}\colon \cW(T^*S^{n-1})\to\cW(T^*S^{n-1})$, or equivalently, the dg functors
	\begin{itemize}
		\item $R_0\colon \cA_1(3)\amalg\cA_1(4)\to\cA_1(3)\amalg\cA_1(4)$ when $n=1$, where $\cA_1(i)$ is the semifree dg category with a single object $K_i$ with no generating morphisms,
		
		\item $R_1\colon\cC_1[z^{-1}]\to\cC_1[z^{-1}]$ when $n=2$,
		
		\item $R_{n-1}\colon\cC_{n-1}\to\cC_{n-1}$ when $n\geq3$.
	\end{itemize}
	There is the induced morphism between diagrams
	\begin{equation}\label{eq:sphere-reflection-diagram}\begin{tikzcd}
		\cA_1(1)\ar[d,"1", blue] & \cC_{n-1}\ar[l]\ar[r]\ar[d,"R_{n-1}", blue] & \cA_1(2)\ar[d,"1", blue]\\
		\cA_1(1) & \cC_{n-1}\ar[l]\ar[r] & \cA_1(2)
	\end{tikzcd}\end{equation}
	for $n\geq3$ (replace $\cC_{n-1}$ and $R_{n-1}$ with $\cC_1[z^{-1}]$ and $R_1$ when $n=2$, and with $\cA_1(3)\amalg\cA_1(4)$ and $R_0$ when $n=1$), and the homotopy colimit functor sends it to $R_n\colon\cC_n\to\cC_n$ when $n\geq 2$ ($R_1\colon\cC_1[z^{-1}]\to\cC_1[z^{-1}]$ when $n=1$). So, using the homotopy colimit functor described in Theorem \ref{thm:hocolim-functor-dg} and Theorem \ref{thm:hocolim-functor-dg-loc}, we can determine $R_n$ inductively.
	
	\textbf{Case $n=1$:} There is the pretriangulated equivalence $\cW(T^*S^0)\simeq\cA_1(3)\amalg\cA_1(4)$. Then, it is easy to see that the reflection map $r_0\colon T^*S^0\to T^*S^0$ induces the dg functor
	\begin{align*}
		R_0\colon\cA_1(3)\amalg\cA_1(4) &\to \cA_1(3)\amalg\cA_1(4) \\
		K_3 &\mapsto K_4,\quad K_4 \mapsto K_3 .
	\end{align*}
	We note that $K_3$ may be actually mapped to $K_4[m]$ for some $m\in\Z$, and $K_4$ to $K_3[-m]$, however, we can rename $K_4[m]$ as $K_4$ in that case.
	By the proof of Proposition \ref{prp:wfuk-sphere}, we have the quasi-equivalence
	\[\hocolim(\cA_1(1)\leftarrow \cA_1(3)\amalg\cA_1(4)\rightarrow \cA_1(2))\simeq\cC_1'[\{t_{K_3},t_{K_4}\}^{-1}]\]
	where $\cC_1'$ is defined in the proof. Then, by applying the homotopy colimit functor to the morphism of diagrams in (\ref{eq:sphere-reflection-diagram}) for $n=1$, Theorem \ref{thm:hocolim-functor-dg} gives us
	\begin{align*}
		R_1\colon\cC_1'[\{t_{K_3},t_{K_4}\}^{-1}]&\to\cC_1'[\{t_{K_3},t_{K_4}\}^{-1}]\\
		K_1\mapsto K_1,\quad K_2&\mapsto K_2,\quad t_{K_3}\mapsto t_{R_0(K_3)}=t_{K_4},\quad t_{K_4}\mapsto t_{R_0(K_4)}=t_{K_3},\\
		t'_{K_3},\hat t_{K_3},\check t_{K_3},\bar t_{K_3}& \mapsto t'_{K_4},\hat t_{K_4},\check t_{K_4},\bar t_{K_4}  &&\text{respectively,}\\
		t'_{K_4},\hat t_{K_4},\check t_{K_4},\bar t_{K_4}& \mapsto t'_{K_3},\hat t_{K_3},\check t_{K_3},\bar t_{K_3}&&\text{respectively} .
	\end{align*}
	Finally, we need to interpret this functor on $\cC_1[z^{-1}]$ using the quasi-equivalence given in the proof of Proposition \ref{prp:wfuk-sphere}:
	\begin{align*}
		\cC_1'[\{t_{K_3},t_{K_4}\}^{-1}]&\xrightarrow{\sim}\cC_1[z^{-1}]\\
		K_1\mapsto L,\quad K_2&\mapsto L,\quad t_{K_3}\mapsto z,\quad t_{K_4}\mapsto 1_L\\
		\text{(implies $t'_{K_3}$}&\mapsto z'\text{and } t'_{K_4}\mapsto 1_L)
	\end{align*}
	It has a quasi-inverse
	\begin{align*}
		\label{eq:s1-quasi-equivalence}\cC_1[z^{-1}] &\xrightarrow{\sim} \cC_1'[\{t_{K_3},t_{K_4}\}^{-1}]\\
		\notag L&\mapsto K_1,\quad z\mapsto t_{K_4}'t_{K_3} .
	\end{align*}
	This gives us the dg functor
	\begin{alignat*}{3}
		R_1\colon\cC_1[z^{-1}]&\xrightarrow{\sim} \cC_1'[\{t_{K_3},t_{K_4}\}^{-1}] &&\xrightarrow{R_1} \cC_1'[\{t_{K_3},t_{K_4}\}^{-1}] &&\xrightarrow{\sim}\cC_1[z^{-1}]\\
		L&\mapsto\hspace{1cm} K_1 &&\mapsto \hspace{1cm}K_1 &&\mapsto L\\
		z&\mapsto\hspace{1cm} t'_{K_4}t_{K_3} &&\mapsto \hspace{1cm}t'_{K_3}t_{K_4} &&\mapsto z'
	\end{alignat*}
	as desired.
		
	\textbf{Case $n=2$:} We know that $\cW(T^*S^1)\simeq\cC_1[x^{-1}]$ (note that we replaced $z$ with $x$ for notational purposes), and from the case $n=1$, we have the functor $R_1\colon\cC_1[x^{-1}]\to\cC_1[x^{-1}]$ such that $R_1(L)=L$ and $R_1(x)=x'$.
	By the proof of Proposition \ref{prp:wfuk-sphere}, we have the quasi-equivalence
	\[\hocolim(\cA_1(1)\xleftarrow{\alpha} \cC_1[x^{-1}]\xrightarrow{\beta} \cA_1(2)) \simeq\cC_2'[t_L^{-1}]\]
	where $\cC_2'$ is defined in the proof. Then, by applying the homotopy colimit functor to the morphism of diagrams in (\ref{eq:sphere-reflection-diagram}) for $n=2$, Theorem \ref{thm:hocolim-functor-dg-loc} gives us
	\begin{align*}
		R_2\colon\cC_2'[t_L^{-1}]&\to\cC_2'[t_L^{-1}]\\
		K_1\mapsto K_1,\quad K_2&\mapsto K_2,\quad t_{L}\mapsto t_{R_1(L)}=t_{L},\quad t_{x}\mapsto t_{R_1(x)}=t_{x'},\\
		t'_{L},\hat t_{L},\check t_L,\bar t_L& \mapsto t'_{L},\hat t_{L},\check t_L,\bar t_L  &&\text{respectively,}
	\end{align*}
	where $t_{x'}$ is determined by the formula in (\ref{eq:t-theta-loc}) (after replacing $i_1$ with $\alpha$, $i_2$ with $\beta$). It is given by
	\begin{align*}
		t_{x'}&= -\beta(x')\circ t_x\circ \alpha(x')-\beta(\hat x)\circ t_L\circ \alpha(x')+ \beta(x')\circ t_L\circ \alpha(\check x)+\beta(\hat x\circ x'-x'\circ\check x)\circ t_L\\
		&=-1_{K_2}\circ t_x\circ 1_{K_1} - 0\circ t_L\circ 1_{K_1} + 1_{K_2}\circ t_L\circ 0 + 0\circ t_L\\
		&=-t_x ,
	\end{align*}
	since $\alpha(x)=\alpha(x')=1_{K_1}$, $\beta(x)=\beta(x')=1_{K_2}$, and $\alpha(\hat x)=\alpha(\check x)=\beta(\hat x)=\beta(\check x)=0$.
	Finally, we need to interpret this functor on $\cC_2$ using the quasi-equivalence given in the proof of Proposition \ref{prp:wfuk-sphere}:
	\begin{align*}
		\cC_2'[t_L^{-1}]&\xrightarrow{\sim}\cC_2\\
		K_1\mapsto L,\quad K_2&\mapsto L,\quad t_L\mapsto 1_L,\quad t_x\mapsto z\\
		\text{(implies $t'_{L}$}&\mapsto 1_L\text{)}
	\end{align*}
	It has a quasi-inverse
	\begin{align*}
		\label{eq:s2-quasi-equivalence}\cC_2&\xrightarrow{\sim}\cC_2'[t_L^{-1}]\\
		\notag L&\mapsto K_1,\quad z\mapsto t_{L}'t_{x} .
	\end{align*}
	This gives us the dg functor
	\begin{alignat*}{3}
		R_2\colon\cC_2&\xrightarrow{\sim} \cC_2'[t_L^{-1}] &&\xrightarrow{R_2} \cC_2'[t_L^{-1}] &&\xrightarrow{\sim}\cC_2\\
		L&\mapsto\hspace{0.5cm} K_1 &&\mapsto \hspace{0.5cm}K_1 &&\mapsto L\\
		z&\mapsto\hspace{0.2cm} t'_{L}t_{x} &&\mapsto \hspace{0.2cm}-t'_{L}t_{x} &&\mapsto-z
	\end{alignat*}
	as desired.
	
	\textbf{Case $n\geq 3$:} We know that $\cW(T^*S^{n-1})\simeq\cC_{n-1}$ (note that we replaced $z$ in $\cC_{n-1}$ with $x$ for notational purposes), and by the induction hypothesis, we have the functor $R_{n-1}\colon\cC_{n-1}\to\cC_{n-1}$ such that $R_{n-1}(L)=L$ and $R_{n-1}(x)=-x$.
	By the proof of Proposition \ref{prp:wfuk-sphere}, we have the quasi-equivalence
	\[\hocolim(\cA_1(1)\xleftarrow{\alpha} \cC_{n-1}\xrightarrow{\beta} \cA_1(2)) \simeq\cC_n'[t_L^{-1}]\]
	where $\cC_n'$ is defined in the proof. Then, by applying the homotopy colimit functor to the morphism of diagrams in (\ref{eq:sphere-reflection-diagram}) for $n\geq 3$, Theorem \ref{thm:hocolim-functor-dg} gives us
	\begin{align*}
		R_n\colon\cC_n'[t_L^{-1}]&\to\cC_n'[t_L^{-1}]\\
		K_1\mapsto K_1,\quad K_2&\mapsto K_2,\quad t_{L}\mapsto t_{R_{n-1}(L)}=t_{L},\quad t_{x}\mapsto t_{R_{n-1}(x)}=t_{-x}=-t_x,\\
		t'_{L},\hat t_{L},\check t_L,\bar t_L& \mapsto t'_{L},\hat t_{L},\check t_L,\bar t_L  &&\text{respectively,}
	\end{align*}
	where $t_{-x}=-t_x$ is determined by the formula in Definition \ref{dfn:generalised-t-morphisms}.
	Finally, we need to interpret this functor on $\cC_n$ using the quasi-equivalence given in the proof of Proposition \ref{prp:wfuk-sphere}:
	\begin{align*}
		\cC_n'[t_L^{-1}]&\xrightarrow{\sim}\cC_n\\
		K_1\mapsto L,\quad K_2&\mapsto L,\quad t_L\mapsto 1_L,\quad t_x\mapsto z\\
		\text{(implies $t'_{L}$}&\mapsto 1_L\text{)}
	\end{align*}
	It has a quasi-inverse
	\begin{align*}
		\cC_n&\xrightarrow{\sim}\cC_n'[t_L^{-1}]\\
		L&\mapsto K_1,\quad z\mapsto t_{L}'t_{x} .
	\end{align*}
	This gives us the dg functor
	\begin{alignat*}{3}
		R_n\colon\cC_n&\xrightarrow{\sim} \cC_n'[t_L^{-1}] &&\xrightarrow{R_n} \cC_n'[t_L^{-1}] &&\xrightarrow{\sim}\cC_n\\
		L&\mapsto\hspace{0.5cm} K_1 &&\mapsto \hspace{0.5cm}K_1 &&\mapsto L\\
		z&\mapsto\hspace{0.2cm} t'_{L}t_{x} &&\mapsto \hspace{0.2cm}-t'_{L}t_{x} &&\mapsto-z
	\end{alignat*}
	as desired. This concludes the proof.
\end{proof}

\bibliographystyle{amsalpha}
\bibliography{hocolim-functor}

\end{document}